\documentclass[11pt]{amsart}

\usepackage{graphics}
\usepackage{color}
\usepackage[a4paper, margin=3cm]{geometry}

\usepackage{amssymb,enumerate}
\usepackage{amsmath}
\allowdisplaybreaks
\usepackage{bbm}           
\usepackage{bm}
\usepackage{eso-pic,graphicx}
\usepackage{tikz}
\usepackage{cite}
\usepackage{esint}
\usepackage[colorlinks=true, pdfstartview=FitV, linkcolor=blue, citecolor=blue, urlcolor=blue]{hyperref}

\usepackage{graphicx}

\usepackage{amsmath,amsfonts}
\usepackage{rotating}
\usepackage{amssymb}
\usepackage{verbatim}
\usepackage{rotating}
\usepackage{mathrsfs}
\DeclareSymbolFont{largesymbols}{OMX}{cmex}{m}{n}
\makeatletter
\def\Ddots{\mathinner{\mkern1mu\raise\p@
\vbox{\kern7\p@\hbox{.}}\mkern2mu
\raise4\p@\hbox{.}\mkern2mu\raise7\p@\hbox{.}\mkern1mu}}
\makeatother
\newtheorem*{thA}{Theorem A}
\newtheorem*{thB}{Theorem B}

\def\XXint#1#2#3{{\setbox0=\hbox{$#1{#2#3}{\int}$}
\vcenter{\hbox{$#2#3$}}\kern-.5\wd0}}

\begin{document}

\newtheorem{definition}{Definition}
\newtheorem{theorem}[definition]{Theorem}
\newtheorem{proposition}[definition]{Proposition}
\newtheorem{conjecture}[definition]{Conjecture}
\def\theconjecture{\unskip}
\newtheorem{corollary}[definition]{Corollary}
\newtheorem{lemma}[definition]{Lemma}
\newtheorem{claim}[definition]{Claim}
\newtheorem{sublemma}[definition]{Sublemma}
\newtheorem{observation}[definition]{Observation}
\theoremstyle{definition}

\newtheorem{notation}[definition]{Notation}
\newtheorem{remark}[definition]{Remark}
\newtheorem{question}[definition]{Question}

\newtheorem{example}[definition]{Example}
\newtheorem{problem}[definition]{Problem}
\newtheorem{exercise}[definition]{Exercise}
 \newtheorem{thm}{Theorem}
 \newtheorem{cor}[thm]{Corollary}
 \newtheorem{lem}{Lemma}[section]
 \newtheorem{prop}[thm]{Proposition}
 \theoremstyle{definition}
 \newtheorem{dfn}[thm]{Definition}
 \theoremstyle{remark}
 \newtheorem{rem}{Remark}
 \newtheorem{ex}{Example}
 
\def\C{\mathbb{C}}
\def\R{\mathbb{R}}
\def\Rn{{\mathbb{R}^n}}
\def\Rns{{\mathbb{R}^{n+1}}}
\def\Sn{{{S}^{n-1}}}
\def\M{\mathbb{M}}
\def\N{\mathbb{N}}
\def\Q{{\mathbb{Q}}}
\def\Z{\mathbb{Z}}
\def\X{\mathbb{X}}
\def\Y{\mathbb{Y}}
\def\F{\mathcal{F}}
\def\L{\mathcal{L}}
\def\S{\mathcal{S}}
\def\supp{\operatorname{supp}}
\def\essi{\operatornamewithlimits{ess\,inf}}
\def\esss{\operatornamewithlimits{ess\,sup}}

\numberwithin{equation}{section}
\numberwithin{thm}{section}
\numberwithin{definition}{section}

\def\earrow{{\mathbf e}}
\def\rarrow{{\mathbf r}}
\def\uarrow{{\mathbf u}}
\def\varrow{{\mathbf V}}
\def\tpar{T_{\rm par}}
\def\apar{A_{\rm par}}

\def\reals{{\mathbb R}}
\def\torus{{\mathbb T}}
\def\t{{\mathcal T}}
\def\heis{{\mathbb H}}
\def\integers{{\mathbb Z}}
\def\z{{\mathbb Z}}
\def\naturals{{\mathbb N}}
\def\complex{{\mathbb C}\/}
\def\distance{\operatorname{distance}\,}
\def\support{\operatorname{support}\,}
\def\dist{\operatorname{dist}\,}
\def\Span{\operatorname{span}\,}
\def\degree{\operatorname{degree}\,}
\def\kernel{\operatorname{kernel}\,}
\def\dim{\operatorname{dim}\,}
\def\codim{\operatorname{codim}}
\def\trace{\operatorname{trace\,}}
\def\Span{\operatorname{span}\,}
\def\dimension{\operatorname{dimension}\,}
\def\codimension{\operatorname{codimension}\,}
\def\nullspace{\scriptk}
\def\kernel{\operatorname{Ker}}
\def\ZZ{ {\mathbb Z} }
\def\p{\partial}
\def\rp{{ ^{-1} }}
\def\Re{\operatorname{Re\,} }
\def\Im{\operatorname{Im\,} }
\def\ov{\overline}
\def\eps{\varepsilon}
\def\lt{L^2}
\def\diver{\operatorname{div}}
\def\curl{\operatorname{curl}}
\def\etta{\eta}
\newcommand{\norm}[1]{ \|  #1 \|}
\def\expect{\mathbb E}
\def\bull{$\bullet$\ }

\def\blue{\color{blue}}
\def\red{\color{red}}

\def\xone{x_1}
\def\xtwo{x_2}
\def\xq{x_2+x_1^2}
\newcommand{\abr}[1]{ \langle  #1 \rangle}

\newcommand{\Norm}[1]{ \left\|  #1 \right\| }
\newcommand{\set}[1]{ \left\{ #1 \right\} }
\newcommand{\ifou}{\raisebox{-1ex}{$\check{}$}}
\def\one{\mathbf 1}
\def\whole{\mathbf V}
\newcommand{\modulo}[2]{[#1]_{#2}}
\def \essinf{\mathop{\rm essinf}}
\def\scriptf{{\mathcal F}}
\def\scriptg{{\mathcal G}}
\def\m{{\mathcal M}}
\def\scriptb{{\mathcal B}}
\def\scriptc{{\mathcal C}}
\def\scriptt{{\mathcal T}}
\def\scripti{{\mathcal I}}
\def\scripte{{\mathcal E}}
\def\V{{\mathcal V}}
\def\scriptw{{\mathcal W}}
\def\scriptu{{\mathcal U}}
\def\scriptS{{\mathcal S}}
\def\scripta{{\mathcal A}}
\def\scriptr{{\mathcal R}}
\def\scripto{{\mathcal O}}
\def\scripth{{\mathcal H}}
\def\scriptd{{\mathcal D}}
\def\scriptl{{\mathcal L}}
\def\scriptn{{\mathcal N}}
\def\scriptp{{\mathcal P}}
\def\scriptk{{\mathcal K}}
\def\frakv{{\mathfrak V}}
\def\v{{\mathcal V}}
\def\C{\mathbb{C}}
\def\D{\mathcal{D}}
\def\R{\mathbb{R}}
\def\Rn{{\mathbb{R}^n}}
\def\rn{{\mathbb{R}^n}}
\def\Rm{{\mathbb{R}^{2n}}}
\def\r2n{{\mathbb{R}^{2n}}}
\def\Sn{{{S}^{n-1}}}
\def\bbM{\mathbb{M}}
\def\N{\mathbb{N}}
\def\Q{{\mathcal{Q}}}
\def\Z{\mathbb{Z}}
\def\F{\mathcal{F}}
\def\L{\mathcal{L}}
\def\G{\mathscr{G}}
\def\ch{\operatorname{ch}}
\def\supp{\operatorname{supp}}
\def\dist{\operatorname{dist}}
\def\essi{\operatornamewithlimits{ess\,inf}}
\def\esss{\operatornamewithlimits{ess\,sup}}
\def\dis{\displaystyle}
\def\dsum{\displaystyle\sum}
\def\dint{\displaystyle\int}
\def\dfrac{\displaystyle\frac}
\def\dsup{\displaystyle\sup}
\def\dlim{\displaystyle\lim}
\def\bom{\Omega}
\def\om{\omega}

\author[B. Dan]{BinWei Dan}
\address{BinWei Dan:
	School of Mathematical Sciences \\
	Beijing Normal University \\
	Laboratory of Mathematics and Complex Systems \\
	Ministry of Education \\
	Beijing 100875 \\
	People's Republic of China}
\email{bwdan@mail.bnu.edu.cn}

\author[M. Qin]{Moyan Qin}
\address{Moyan Qin:
	School of Mathematical Sciences \\
	Beijing Normal University \\
	Laboratory of Mathematics and Complex Systems \\
	Ministry of Education \\
	Beijing 100875 \\
	People's Republic of China}
\email{myqin@bnu.edu.cn}

\author[Q. Xue]{Qingying Xue$^{*}$}
\address{Qingying Xue:
	School of Mathematical Sciences \\
	Beijing Normal University \\
	Laboratory of Mathematics and Complex Systems \\
	Ministry of Education \\
	Beijing 100875 \\
	People's Republic of China}
\email{qyxue@bnu.edu.cn}


\thanks{The authors were partly supported by the National Key R\&D Program of China (No. 2020YFA0712900) and NNSF of China (No. 12271041).
\thanks{$^{*}$ Corresponding author, e-mail address: qyxue@bnu.edu.cn}}

\date{\today}
\title[BILINEAR ROUGH SINGULAR INTEGRALS UNDER A FRACTIONAL GEOMETRIC CONDITION]
{\bf Bilinear rough singular integrals under a fractional geometric condition}

\begin{abstract}
	
		We establish the Banach-range boundedness of bilinear rough singular integral operators, together with their maximal and maximally truncated forms, under the fractional geometric condition on the mean-zero angular kernel
		\[
		\sup_{\xi \in \mathbb{S}^{1}}\int_{\mathbb{S}^{1}} \frac{|\Omega(\theta)|}{|\theta \cdot \xi|^{a}} \, d\sigma(\theta) < \infty, \qquad \frac12 < a < 1.
		\]
		This condition imposes integrability strictly weaker than the $L^q(\mathbb{S}^1) (q>1)$  constraints considered by Grafakos, He, Honz\'ik (Adv. Math., 2018), Dosidis and Slavíková (Math. Ann., 2024), while defining a class of functions that is neither contained in nor contains the classical Orlicz space $L(\log L)^\alpha(\mathbb{S}^1) $ ($\alpha>1$). Our proof avoids traditional wavelet decompositions of the multiplier, instead using local Fourier series expansions of the input functions.

\end{abstract}\maketitle

\section{\bf Introduction}

It was well-known that the theory of rough singular integrals originated  in the fundamental work of Calder\'on and Zygmund~\cite{calderon_existence_1952}, who proved that the linear operator 
$	\mathcal {T}_\Omega(f)(x) = \text{p.v.} \int_{\mathbb{R}^{n}} \frac{\Omega((x-y)/|x-y|)}{|x-y|^n} f( y)\, dy$
is bounded on $L^p(\mathbb{R}^n)$ for $1 < p < \infty$, provided that $\Omega \in L \log L(\mathbb{S}^{n-1})$ and satisfies a mean zero condition~\cite{calderon_singular_1956}. This result was later improved under weaker assumptions, such as $\Omega \in H^1(\mathbb{S}^{n-1})$, by Coifman and Weiss~\cite{coifman_extensions_1977} and Connett~\cite{connett_singular_1979}. Subsequently, the $L^p(\mathbb{R}^n)$ boundedness was reconsidered by Grafakos and Stefanov  \cite{grafakos1998lp} with  a logarithmic kernel condition of the form
\begin{equation}\label{grafakos_cond}
	\sup_{\xi \in \mathbb{S}^{n-1}} \int_{\mathbb{S}^{n-1}} |\Omega(\theta)|
	\left(\ln \frac{1}{|\theta \cdot \xi|}\right)^{1+\alpha} \, d\sigma(\theta) < \infty, \quad \alpha > 0.
\end{equation}
This condition has no including relationship with $ H^1(\mathbb{S}^{n-1})$. 
The weak-type $(1,1)$ boundedness of $\mathcal {T}_\Omega$ was established in low dimensions by Christ and Rubio de Francia~\cite{christ_weak_1988-1} for $n\le 5 $ and Hofmann~\cite{hofmann_weighted_1988} for $n=2$ , and extended to all dimensions by Seeger~\cite{seeger_singular_1996} and to more general settings by Tao~\cite{Tao1999}. 

Building upon these linear foundations, attention naturally shifted to the \emph{bilinear} analogues. The associated bilinear singular integral operator is defined by
\begin{equation}\label{T}
\mathcal	{T}_\Omega(f_1, f_2)(x) = \text{p.v.} \int_{\mathbb{R}^{2n}} \frac{\Omega((y_1, y_2)/|(y_1, y_2)|)}{|(y_1, y_2)|^2} f_1(x - y_1) f_2(x - y_2)\, dy.
\end{equation}
The boundedness of $\mathcal {T}_\Omega \colon L^{p_1}(\mathbb{R}^n) \times L^{p_2}(\mathbb{R}^n) \to L^p(\mathbb{R}^n)$ was first established by Coifman and Meyer~\cite{coifman_commutators_1975} for smooth kernels when the indexs  $(p_1,p_2,p)\in\mathcal{H}^1$, where 
\[\mathcal{H}^q=\left\{(p_1,p_2,p):\;1<p_1,p_2<\infty,\;\frac{1}{2}<p<\infty,\;\frac{1}{p_1}+\frac{1}{p_2}=\frac{1}{p},\;\frac{1}{p}+\frac{1}{q}<2\right\},1\le q\le \infty.\]Subsequent extensions to broader exponent ranges, including the quasi-Banach regime $p < 1$, were obtained by Kenig and Stein~\cite{Kenig1999} and Grafakos and Torres~\cite{Grafakos2002}. 

For kernels where $\Omega$ lacks smoothness, bilinear bounds were initially obtained in one dimension in \cite{diestel2011method}, and later generalized to higher dimensions under various integrability conditions~in \cite{grafakos2018rough, grafakos2020l2, he2023improved}. In particular, Grafakos, He, Honz\'ik \cite{grafakos2018rough} 
showed that $\mathcal {T}_\Omega$ is bounded from  $L^{p_1}(\mathbb{R}^n) \times L^{p_2}(\mathbb{R}^n) $ to $ L^p(\mathbb{R}^n)$ if $\Omega\in L^\infty(\mathbb{S}^{2n-1})$. This result was later extended to $\Omega\in L^q(\mathbb{S}^{2d-1})$ for all $q>\frac{4}{3}$ in \cite{grafakos2020l2}.
In \cite{GHS2019}, the authors showed that $\mathcal{H}^q$ is the largest open range for which the operator $\mathcal {T}_\Omega$ is bounded for $\Omega\in L^q (1<q\leq\infty)$.
Recently, Dosidis and Slavíková~\cite{dosidis2024multilinear} completely characterized the optimal exponent range for the $L^{p_1}(\mathbb{R}^n) \times L^{p_2}(\mathbb{R}^n) \to L^p(\mathbb{R}^n)$  boundedness of $\mathcal {T}_\Omega$ if $\Omega \in L^q(\mathbb{S}^{2n-1})(q>1)$ for $(p_1,p_2,p)\in\mathcal{H}^q(1<q\le \infty)$.
A further breakthrough near the critical integrability endpoint was obtained by Dosidis, Park, and Slavíková~\cite{dosidis2026bilinear}. In their work, the angular kernel $\Omega$ is not assumed to belong to $L^q(\mathbb{S}^{2n-1})$ for any $q>1$; instead, it satisfies the Orlicz-type condition
\begin{equation}\label{E:omega}
\int_{\mathbb{S}^{2n-1}} |\Omega(\theta)| \bigl(\log(e + |\Omega(\theta)|)\bigr)^\alpha \, d\nu(\theta) < \infty, \qquad \text{for some } \alpha \ge 0.
\end{equation}
We then write $\Omega \in L(\log L)^\alpha(\mathbb{S}^{2n-1})$ whenever condition~\eqref{E:omega} holds. 

Still more recently, Bhojak and Shrivastava~\cite{bhojak2025alternate} obtained the following results, by utilizing local Fourier series expansions rather than traditional wavelet decompositions.  \begin{thA}\label{Thm:RoughBanach}
	Let $1\leq p_1,p_2,p<\infty$ with $\frac{1}{p_1}+\frac{1}{p_2}=\frac{1}{p}$. Let $A=A(p_1,p_2,p)$ be defined as
	\[\alpha=1+\max\left\{\frac{1}{p_1},\frac{1}{p_2},\frac{1}{p'}\right\}.\]
	Then, for all $\Omega\in L(LogL)^{\alpha}(\mathbb{S}^{1})$ with $\int_{\mathbb{S}^{1}}\Omega(\theta) d\sigma(\theta)=0$, it holds that
	\[\|\mathcal {T}_\Omega(f_1,f_2)\|_{L^{p}(\R)}\lesssim\|\Omega\|_{L(LogL)^{\alpha}(\mathbb{S}^{1})}\|f_1\|_{L^{p_1}(\R)}\|f_2\|_{L^{p_2}(\R)}.\]
\end{thA}
\begin{thB}
	Let $1<q\leq\infty$ and $(p_1,p_2,p)\in\mathcal{H}^q$. Then, for all $\Omega\in L^q(\mathbb{S}^{1})$ with $\int_{\mathbb{S}^{1}}\Omega(\theta) d\sigma(\theta)=0$, it holds that
\[\|\mathcal {T}_\Omega(f_1,f_2)\|_{L^p(\R)}\lesssim\|\Omega\|_{L^q(\mathbb{S}^{1})}\|f_1\|_{L^{p_1}(\R)}\|f_2\|_{L^{p_2}(\R)}.\]
\end{thB}
However, the boundedness of bilinear singular integrals under the logarithmic directional integrability condition \eqref{grafakos_cond} of Grafakos and Stefanov \textbf{remains an open problem} in Calder\'on-Zygmund theory. In this paper, we make a big step and consider a slightly stronger, power-law type geometric singularity. In the one-dimensional setting, our main results assume the standard vanishing mean condition
\begin{equation}\label{cancellation}
\int_{\mathbb{S}^{1}} \Omega(\theta) \, d\sigma(\theta) = 0,
\end{equation}
and the fractional directional condition
\begin{equation}\label{alpha}
\|\Omega\|_{\mathcal{K}_a} := \sup_{\xi \in \mathbb{S}^{1}} \int_{\mathbb{S}^{1}} \frac{|\Omega(\theta)|}{|\theta \cdot \xi|^{a}} \, d\sigma(\theta) < \infty,\qquad 0<a<1,
\end{equation}
where $\mathcal{K}_a(\mathbb{S}^{1})$ denotes the space of measurable functions satisfying~\eqref{alpha}.

To clarify the relationship between our framework and previous works, let $GS_\alpha(\mathbb{S}^{1})$ denote the class of kernels satisfying Condition \eqref{grafakos_cond}.  Then, the following including relationships hold
\[
L^q(\mathbb{S}^1) (q>1/(1-a))\subsetneq
\mathcal{K}_a(\mathbb{S}^{1})\subsetneq GS_\alpha(\mathbb{S}^{1})\, \quad
L^q(\mathbb{S}^1) (q>1)\subsetneq L(\log L)^\alpha(\mathbb{S}^1).
\]
Moreover, our fractional conditiopn is strictly incomparable with this logarithmic scale
\[
\mathcal{K}_a(\mathbb{S}^1) \not\subset L(\log L)^\alpha(\mathbb{S}^1) \quad \text{and} \quad L(\log L)^\alpha(\mathbb{S}^1) \not\subset \mathcal{K}_a(\mathbb{S}^1).
\]This will be demonstrated in Theorem \ref{thmincomparability} below.

We now state our main theorem as follows.

\begin{theorem}\label{thmmain}
	Let $1 < p_1, p_2, p < \infty$ with $1/p = 1/p_1 + 1/p_2$.
	Suppose that $\Omega$ satisfies condition~\eqref{cancellation} and	condition~\eqref{alpha} for $1/2 < a < 1$.
	Then there exists a constant $C > 0$ such that
	\[
	\|\mathcal {T}_\Omega(f_1, f_2)\|_{L^p(\mathbb{R})} \leq C\|\Omega\|_{\mathcal{K}_a}\|f_1\|_{L^{p_1}(\mathbb{R})}\|f_2\|_{L^{p_2}(\mathbb{R})}.
	\]
\end{theorem}

Furthermore, we extend our analysis to address the corresponding maximal variants of these operators, which historically require more complex pointwise controls (see, for instance, \cite{BH19, Par25}). Recall that, the maximally truncated bilinear operator $\mathcal {T}_\Omega^*$ is defined by
\[
\mathcal {T}_\Omega^*(f_1, f_2)(x) = \sup_{\epsilon>0} \left| \int_{|(y_1, y_2)|
	> \epsilon} \frac{\Omega((y_1, y_2)/|(y_1, y_2)|)}{|(y_1, y_2)|^2} f_1(x-y_1) f_2(x-y_2) \, dy_1 dy_2 \right|.
\]
The boundedness of smooth maximal multilinear singular integrals was initially established by Grafakos and Torres~\cite{GT2002_maximal}. For rough kernels, Buri\'ankov\'a and Honz\'ik~\cite{BH19} first obtained the $L^2 \times L^2 \to L^1$ boundedness of the bilinear maximal operator when $\Omega \in L^2(\mathbb{S}^{2n-1})$, alongside $L^{p_1} \times L^{p_2} \to L^p$ bounds for $\Omega \in L^\infty(\mathbb{S}^{2n-1})$. This initial $L^2$ result was later extended to general multilinear settings by Grafakos, He, Honz\'ik, and Park~\cite{GHHP2024}. Recently, a significant advancement was made by Park~\cite{Par25}, who completely extended these maximal estimates to the full range of exponents under the much weaker $L^q(\mathbb{S}^1) $ ($q>1$) condition.
Following this, Bhojak and Shrivastava~\cite{bhojak2025alternate} successfully applied their alternative framework to establish bounds for the maximally truncated bilinear operators in the one-dimensional setting as in Theorem B. 

Building on the above works, we establish the following boundedness of $\mathcal {T}_\Omega^*$.

\begin{theorem}\label{thmtrunc}
	Let $1 < p_1, p_2, p < \infty$ with $1/p = 1/p_1 + 1/p_2$.
	Suppose that $\Omega$ satisfies condition~\eqref{cancellation} and condition~\eqref{alpha} with $1/2 < a < 1$.
	Then there exists a constant $C > 0$ such that
	\[
	\|\mathcal {T}_\Omega^*(f_1, f_2)\|_{L^p(\mathbb{R})} \le C \|\Omega\|_{\mathcal{K}_a} \|f_1\|_{L^{p_1}(\mathbb{R})} \|f_2\|_{L^{p_2}(\mathbb{R})}.
	\]
\end{theorem}

In order to prove the boundedness of $\mathcal {T}_\Omega^*$, it is necessary to consider the associated maximal operator $M_\Omega$, defined by
\[
M_\Omega(f_1, f_2)(x) = \sup_{R>0} \frac{1}{R^2} \int_{|(y_1, y_2)|
	\le R} \left| \Omega\left(\frac{(y_1, y_2)}{|(y_1, y_2)|}\right) \right| |f_1(x-y_1) f_2(x-y_2)| \, dy_1 dy_2.
\]

Note that, the boundedness of $M_\Omega$ in the bilinear and multilinear cases was previously obtained by Buri\'ankov\'a and Honz\'ik~\cite{BH19}, Grafakos, He, Honz\'ik, and Park~\cite{GHHP2024}, and Bhojak and Shrivastava~\cite{bhojak2025alternate}. 
We obtain the following result for $M_\Omega$.

\begin{theorem}\label{thmmax}
	Let $1 < p_1, p_2, p < \infty$ with $1/p = 1/p_1 + 1/p_2$.
	Suppose that $\Omega$ satisfies condition~\eqref{alpha} with $1/2 < a < 1$.
	Then there exists a constant $C > 0$ such that
	\[
	\|M_\Omega(f_1, f_2)\|_{L^p(\mathbb{R})} \le C \|\Omega\|_{\mathcal{K}_a} \|f_1\|_{L^{p_1}(\mathbb{R})} \|f_2\|_{L^{p_2}(\mathbb{R})}.
	\]
\end{theorem}

Having established the boundedness of the bilinear operators and their maximal variants under the fractional condition \eqref{alpha}, it is natural to demonstrarte how this framework relates to the recent advancements governed by the Orlicz spaces $L(\log L)^\alpha(\mathbb{S}^{1})$. While the Orlicz space condition imposes an isotropic constraint on the overall global size of the kernel's singularities, our condition \eqref{alpha} imposes an anisotropic, geometric restriction specifically targeting singularities concentrated near equators. The following theorem makes this distinction precise.

\begin{theorem}\label{thmincomparability}
	Let $0 < a < 1$ and $\alpha > 1$.
	The fractional angular condition~\eqref{alpha} and the Orlicz space $L(\log L)^\alpha(\mathbb{S}^{1})$ are mutually incomparable.
	\begin{enumerate}
		\item[\rm(i)] \textbf{(Escape from $L(\log L)^\alpha$)} There exists an integrable function $\Omega$ on $\mathbb{S}^{1}$ that uniformly satisfies Condition ~\eqref{alpha} but fails to belong to $L(\log L)^\alpha(\mathbb{S}^{1})$.
		\item[\rm(ii)] \textbf{(Failure of the fractional condition)} There exists a function $\Omega \in L(\log L)^\alpha(\mathbb{S}^{1})$ that strictly fails condition~\eqref{alpha}.
	\end{enumerate}
\end{theorem}
The lack of mutual inclusion is not restricted to the one-dimensional setting;
the construction naturally extends to higher dimensions $\mathbb{S}^{n-1}$, as detailed in Remark~\ref{rem_higher_dim} at the end of Section \ref{sec4}.

The main contributions of this paper can be summarized as follows:

\begin{itemize}
	\item \textbf{New fractional condition:} We introduce a fractional directional condition ($\mathcal{K}_a$)  that controls singularities along specific directions. Under this condition, we establish the Banach-range boundedness of bilinear rough singular integral operators, as well as their associated maximal and maximally truncated operators. This approach differs fundamentally from the global size restrictions imposed by classical Orlicz spaces.
	\item \textbf{Alternative approach for full range bounds:} Instead of traditional wavelet decompositions, we use local Fourier series expansions. This approach allows us to establish boundedness across the full range of exponents ($1 < p_1, p_2, p < \infty$) for these bilinear operators and their maximal versions.
	
		\item \textbf{Strict incomparability:} We show that the fractional directional condition ($\mathcal{K}_a$) and the Orlicz space $L(\log L)^\alpha$ are mutually incomparable; that is, neither class contains the other. This demonstrates that our framework captures a genuinely distinct class of rough kernels not covered by classical Orlicz conditions.
		
\end{itemize}

The remainder of the paper is organized as follows. In Section~\ref{sec2}, we prove Theorem~\ref{thmmain} via a local Fourier expansion scheme. Section~\ref{secmaximal} extends these geometric and analytic arguments to the maximal and maximally truncated operators, proving Theorems~\ref{thmtrunc} and~\ref{thmmax}. Finally, Section~\ref{sec4} gives a measure-theoretic analysis of the relation between the fractional condition~\eqref{alpha} and the classical Orlicz spaces $L(\log L)^\alpha(\mathbb{S}^{1})$.

\section{\bf Proof of Theorem \ref{thmmain}}\label{sec2}

In this section, we aim to prove Theorem~\ref{thmmain}.  Some of the ideas in our argument are inspired by the alternative approach of~\cite{bhojak2025alternate}, which is based on local Fourier series expansions of the input functions.

Let us fix a function \( \Omega \) satisfying \( \eqref{alpha} \) with \( 0 < a < 1 \) and zero mean. Following the standard methodology, we utilize a spatial decomposition of the kernel. Let \( \beta \in C_c^\infty(\mathbb{R}) \) be a radial function supported in \( (1/2, 2) \) such that \( \sum_{i \in \mathbb{Z}} \beta_i(t) = \sum_{i \in \mathbb{Z}} \beta(2^i t) = 1 \) for all \( t > 0 \). We define the local kernels and their corresponding operators as
\begin{equation}\label{local_op}
	K^i(y_1, y_2) = \frac{\Omega((y_1, y_2)/|(y_1, y_2)|)}{|(y_1, y_2)|^{2}} \beta_i(|(y_1, y_2)|), \quad i \in \mathbb{Z},
\end{equation}
\[
\mathcal {T}_\Omega^i(f_1, f_2)(x) = \int_{\mathbb{R}^2} K^i(y_1, y_2) f_1(x - y_1) f_2(x - y_2) \, dy_1 dy_2.
\]

For the multi-scale synthesis, we use the Littlewood-Paley decomposition. Let \( \phi \in \mathcal{S}(\mathbb{R}) \) be supported in \( B(0, 2) \) with \( \widehat{\phi}(\xi) = 1 \) for \( |\xi| \le 1 \) and $ 0 \le \widehat{\phi} \le 1$. Define \( \widehat{\psi}(\xi) = \widehat{\phi}(\xi) - \widehat{\phi}(2\xi) \). For $k \in \mathbb{Z}$, we set the scaled functions $\phi_k(x) = 2^k \phi(2^k x)$ and $\psi_k(x) = 2^k \psi(2^k x)$. This yields the following smooth partition of unity
$$\widehat{\phi}(\xi) + \sum_{k \in \mathbb{N}} \widehat{\psi}_k(\xi) = 1 \quad \text{for all } \xi \neq 0.$$
Consequently, the bilinear operator can be decomposed into Low-Low, High-Low, Low-High, and High-High frequency interactions in the way that

\begin{equation}\label{decomp_main}
	\mathcal {T}_\Omega(f_1, f_2) = \mathcal {T}_\Omega^{LL}(f_1, f_2) + \sum_{k \in \mathbb{N}} \left( T_{\Omega, k}^{HL} + T_{\Omega, k}^{LH} + T_{\Omega, k}^{HH} \right)(f_1, f_2),
\end{equation}
which are defined respectively as
\begin{align*}
	\mathcal {T}_\Omega^{LL}(f_1, f_2)(x) &= \sum_{i \in \mathbb{Z}} \mathcal {T}_\Omega^i(f_1 * \phi_i, f_2 * \phi_i)(x), \\
	T_{\Omega, k}^{HL}(f_1, f_2)(x) &= \sum_{i \in \mathbb{Z}} \mathcal {T}_\Omega^i(f_1 * \psi_{i+k}, f_2 * \phi_{i+k-10})(x), \\
	T_{\Omega, k}^{LH}(f_1, f_2)(x) &= \sum_{i \in \mathbb{Z}} \mathcal {T}_\Omega^i(f_1 * \phi_{i+k-10}, f_2 * \psi_{i+k})(x), \\
	T_{\Omega, k}^{HH}(f_1, f_2)(x) &= \sum_{i \in \mathbb{Z}} \sum_{\nu=-10}^{10} \mathcal {T}_\Omega^i(f_1 * \psi_{i+k+\nu}, f_2 * \psi_{i+k})(x).
\end{align*}

The following Proposition constitutes the first technical step in this approach.

\begin{proposition}\label{prop_decay_strong}
	Let \( 1/2 < a < 1 \) and suppose that $\Omega$ satisfies condition~\eqref{alpha}.
	Let \( \mathcal {T}_\Omega^0 \) be the local operator associated with the kernel
	\[
	K^0(y_1,y_2)
	=
	\frac{\Omega\!\left((y_1,y_2)/|(y_1,y_2)|\right)}{|(y_1,y_2)|^2}\,\beta\!\left(|(y_1,y_2)|\right).
	\]
	Then there exists a constant \( c > 0 \) such that for all \( \lambda \ge 1 \),
	\begin{equation}\label{eqT12-Ka}
		\big|\langle \mathcal {T}_\Omega^0(f_1,f_2),f_3\rangle\big|
		\lesssim
		\lambda^{-c}\,
		\|\Omega\|_{\mathcal K_a}\,
		\|f_{l_1}\|_{L^2(\mathbb R)}
		\|f_{l_2}\|_{L^2(\mathbb R)}
		\|f_{l_3}\|_{L^\infty(\mathbb R)},
	\end{equation}
	whenever
	\[
	\operatorname{supp}\widehat f_{l_1}\cup \operatorname{supp}\widehat f_{l_2}
	\subset \{\lambda\le |\xi|\le 2\lambda\},
	\qquad
	\operatorname{supp}\widehat f_{l_3}\subset \{|\xi|\le 2\lambda\},
	\]
	for any permutation $\{l_1,l_2,l_3\}=\{1,2,3\}$.
\end{proposition}

\begin{proof}
	By symmetry, it suffices to treat the configuration
	\[
	\operatorname{supp}\widehat f_1\cup \operatorname{supp}\widehat f_2
	\subset \{\lambda\le |\xi|\le 2\lambda\},
	\qquad
	\operatorname{supp}\widehat f_3\subset \{|\xi|\le 2\lambda\}.
	\]
	As in Section~4 of \cite{bhojak2025alternate}, choose even Schwartz functions $\phi,\psi$ such that
	\[
	\widehat\phi(\xi)=1 \quad \text{for } |\xi|\le \frac{17}{8},
	\qquad
	\operatorname{supp}\widehat\phi\subset \Bigl\{|\xi|\le \frac94\Bigr\},
	\]
	\[
	\widehat\psi(\xi)=1 \quad \text{for } \frac78\le |\xi|\le \frac{17}{8},
	\qquad
	\operatorname{supp}\widehat\psi\subset \Bigl\{\frac34\le |\xi|\le \frac94\Bigr\},
	\]
	and define $\phi_\lambda(x)=\lambda\phi(\lambda x)$, $\psi_\lambda(x)=\lambda\psi(\lambda x)$.
	Then
	\begin{equation}\label{eqfreq-reduction}
		\langle \mathcal {T}_\Omega^0(f_1,f_2),f_3\rangle
		=
		\langle \mathcal {T}_\Omega^0(f_1*\psi_\lambda,\;f_2*\psi_\lambda),\;f_3*\phi_\lambda\rangle
	\end{equation}
	by the Fourier support assumptions.
	
	Following \cite[Section~4.1]{bhojak2025alternate}, we introduce a partition of
	unity $\sum_{m\in\mathbb Z}\eta(x-m)=1$ with $\eta\in C_c^\infty((-1,1))$,
	and an auxiliary cutoff $\widetilde\eta\in C_c^\infty(\mathbb R)$ with
	$\widetilde\eta\equiv 1$ on $[-4,4]$ and $\operatorname{supp}\widetilde\eta
	\subset[-5,5]$. Writing $\eta_m=\eta(\cdot-m)$ and
	$\widetilde\eta_m=\widetilde\eta(\cdot-m)$, the standard localization gives
	\[
	\big|\langle \mathcal{T}_\Omega^0(f_1*\psi_\lambda,f_2*\psi_\lambda),
	f_3*\phi_\lambda\rangle\big|
	\le I_1+I_2+I_3+I_4,
	\]
	where $I_1$, $I_2$, $I_3$ are defined exactly as in
	\cite[Section~4.1]{bhojak2025alternate}, and $I_4$ captures the
	double-commutator term in which both $f_1$ and $f_2$ contribute an error
	\begin{align*}
		I_4 &= \sum_{m\in\mathbb{Z}} \bigg|\bigg\langle T_{\Omega}^0\Big(
		\widetilde{\eta}_{m}^2\big(\widetilde{\eta}_{m}(f_1\ast\psi_{\lambda})
		-(\widetilde{\eta}_{m}f_1)\ast\psi_{\lambda}\big),\;
		\widetilde{\eta}_{m}^2\big(\widetilde{\eta}_{m}(f_2\ast\psi_{\lambda})
		-(\widetilde{\eta}_{m}f_2)\ast\psi_{\lambda}\big)\Big),
		\eta_m(f_3\ast\phi_{\lambda})\bigg\rangle\bigg|.
	\end{align*}
	
	By the mean-value-theorem argument of \cite[Section~4.1]{bhojak2025alternate},
	and using $\|K^0\|_{L^1(\mathbb R^2)}\lesssim\|\Omega\|_{L^1(\mathbb S^1)}
	\le\|\Omega\|_{\mathcal K_a}$ (since $|\theta\cdot\xi|\le 1$ on $\mathbb S^1$),
	one has
	\begin{equation}\label{eqI23-final}
		I_2+I_3
		\lesssim
		\lambda^{-1}\|\Omega\|_{\mathcal K_a}\|f_1\|_2\|f_2\|_2\|f_3\|_\infty.
	\end{equation}
	
	For $I_4$, applying the same pointwise bound
	$|\widetilde\eta_m(f_l*\psi_\lambda)-(\widetilde\eta_m f_l)*\psi_\lambda|
	\lesssim\lambda^{-1}(|f_l|*\Psi_\lambda)$ (with $\Psi(x)=|x||\psi(x)|$)
	simultaneously to both variables $l=1,2$, and then arguing exactly as
	in \cite[Section~4.1]{bhojak2025alternate}, we obtain a double decay:
	\begin{equation}\label{eqI4}
		I_4
		\lesssim
		\lambda^{-2}\|K^0\|_{L^1(\mathbb R^2)}\|f_1\|_2\|f_2\|_2\|f_3\|_\infty
		\lesssim
		\lambda^{-1}\|\Omega\|_{\mathcal K_a}\|f_1\|_2\|f_2\|_2\|f_3\|_\infty.
	\end{equation}

	It remains to bound $I_1$. As in \cite[Section~4.1]{bhojak2025alternate},
	by Cauchy--Schwarz in $m$ and translation invariance, it suffices to
	establish the uniform local bound for $m=0$:
	\begin{equation}\label{eqlocal-goal}
		\Big|\Big\langle \mathcal{T}_\Omega^0\bigl(
		\widetilde\eta^2((\widetilde\eta f_1)*\psi_\lambda),\,
		\widetilde\eta^2((\widetilde\eta f_2)*\psi_\lambda)\bigr),
		\eta(f_3*\phi_\lambda)\Big\rangle\Big|
		\lesssim
		\lambda^{-c}\|\Omega\|_{\mathcal K_a}
		\|\widetilde\eta f_1\|_2\|\widetilde\eta f_2\|_2\|f_3\|_\infty,
	\end{equation}
	which is the content of Lemma~\ref{lemlocal_oscillatory} below.
	Applying it and summing over $m$ gives
	\[
	I_1\lesssim\lambda^{-c}\|\Omega\|_{\mathcal K_a}
	\|f_1\|_2\|f_2\|_2\|f_3\|_\infty.
	\]
	Combined with \eqref{eqI23-final} and \eqref{eqI4}, this proves
	\eqref{eqT12-Ka} in the present frequency configuration. The remaining
	two configurations follow by permuting the roles of $f_1,f_2,f_3$ in
	Lemma~\ref{lemlocal_oscillatory}. This completes the proof.
\end{proof}

\begin{remark}[Scale invariance and uniform decay]\label{rem_scale_inv}
	While Proposition \ref{prop_decay_strong} is stated for the normalized local operator $\mathcal {T}_\Omega^0$, a standard rescaling argument $x \mapsto 2^i x$ demonstrates that analogous bounds hold for $\mathcal {T}_\Omega^i$ uniformly in $i \in \mathbb{Z}$. Under this spatial dilation, an input function with Fourier support in an annulus of radius $\sim 2^{i+k}$ is transformed into a function with Fourier support at radius $\sim 2^k$. Consequently, when applying this proposition to the Littlewood-Paley pieces, we set the frequency parameter to $\lambda \sim 2^k$, which yields a uniform decay rate of $2^{-ck}$ entirely independent of the spatial scale $i$.
\end{remark}

\begin{lemma}\label{lemGN_bound}
	Let $\widetilde\eta \in C_c^\infty(\mathbb R)$ and let $\psi_\lambda(x) = \lambda \psi(\lambda x)$ where $\psi \in \mathcal{S}(\mathbb R)$. For any function $f \in L^2(\mathbb R)$ and $\lambda \ge 1$, define the localized function
	\[
	f_\lambda(x) = \widetilde\eta(x) \big( (\widetilde\eta f) * \psi_\lambda \big)(x).
	\]
	Then we have the uniform bound
	\[
	\|f_\lambda\|_\infty \lesssim \lambda^{1/2} \|\widetilde\eta f\|_2.
	\]
\end{lemma}

\begin{proof}
	Since $f_\lambda$ is compactly supported, the one-dimensional Gagliardo--Nirenberg inequality yields
	\[
	\|f_\lambda\|_\infty \le \sqrt{2} \|f_\lambda\|_2^{1/2} \|f_\lambda'\|_2^{1/2}.
	\]
	By Young's convolution inequality, the base $L^2$ norm satisfies $\|f_\lambda\|_2 \lesssim \|\widetilde\eta f\|_2$.
	For the derivative, the product rule gives
	\[
	f_\lambda'(x) = \widetilde\eta'(x) \big( (\widetilde\eta f) * \psi_\lambda \big)(x) + \widetilde\eta(x) \big( (\widetilde\eta f) * (\psi_\lambda)' \big)(x).
	\]
	Observe that $(\psi_\lambda)'(x) = \frac{d}{dx}[\lambda \psi(\lambda x)] = \lambda^2 \psi'(\lambda x) = \lambda (\psi')_\lambda(x)$.
	Since $\psi$ is a Schwartz function, $\|\psi'\|_1 \lesssim 1$. Applying Young's inequality to both terms, and using $\lambda \ge 1$, we deduce
	\[
	\|f_\lambda'\|_2 \le \|\widetilde\eta'\|_\infty \|\widetilde\eta f\|_2 + \|\widetilde\eta\|_\infty \lambda \|
	(\widetilde\eta f) * (\psi')_\lambda \|_2 \lesssim \lambda \|\widetilde\eta f\|_2.
	\]
	Substituting these $L^2$ bounds back into the Gagliardo--Nirenberg inequality, we strictly obtain
	\[
	\|f_\lambda\|_\infty \lesssim \big( \|\widetilde\eta f\|_2 \big)^{1/2} \big( \lambda \|\widetilde\eta f\|_2 \big)^{1/2} = \lambda^{1/2} \|\widetilde\eta f\|_2.
	\]
	This completes the proof.
\end{proof}

\begin{lemma}\label{lemlocal_oscillatory}
	Let $1/2 < a < 1$ and suppose that $\Omega$ satisfies condition~\eqref{alpha}. Let $\mathcal {T}_\Omega^0$ be the local operator, and let the spatial cutoff functions $\eta, \widetilde{\eta} \in C_c^\infty(\mathbb R)$ as well as the frequency localization functions $\psi_\lambda, \phi_\lambda$ be exactly as defined in the proof of Proposition \ref{prop_decay_strong}. Then there exists a constant $c > 0$ such that for all $\lambda \ge 1$,
	\[
	\Big| \Big\langle \mathcal {T}_\Omega^0\bigl( \widetilde\eta^2((\widetilde\eta f_1)*\psi_\lambda), \widetilde\eta^2((\widetilde\eta f_2)*\psi_\lambda) \bigr), \eta(f_3*\phi_\lambda) \Big\rangle \Big| \lesssim \lambda^{-c}\|\Omega\|_{\mathcal K_a}\|\widetilde\eta f_1\|_2\|\widetilde\eta f_2\|_2\|f_3\|_\infty.
	\]
\end{lemma}

\begin{proof}
	We follow the Fourier-series expansion from Section~4.3 of \cite{bhojak2025alternate}, but we do \emph{not} pass through the auxiliary $(\infty,\infty,\infty)$ inequality from their Section~4.2.
	Define
	\[
	f_{l,\lambda}(x)=
	\begin{cases}
		\widetilde\eta(x)\bigl((\widetilde\eta f_l)*\psi_\lambda\bigr)(x), & l=1,2,\\[1mm]
		\eta(x)(f_3*\phi_\lambda)(x), & l=3.
	\end{cases}
	\]
	As in \cite[Section~4.3]{bhojak2025alternate}, expanding each $f_{l,\lambda}$ 
	into a Fourier series $f_{l,\lambda}(x)=\sum_{k_l\in\mathbb Z}a_{l,k_l}
	e^{-2\pi i\frac{k_l}{10}x}$ and substituting gives
	\[
	\Big\langle \mathcal{T}_\Omega^0\bigl(\widetilde\eta^2((\widetilde\eta f_1)*\psi_\lambda),
	\widetilde\eta^2((\widetilde\eta f_2)*\psi_\lambda)\bigr),\eta(f_3*\phi_\lambda)\Big\rangle
	= \sum_{\vec k\in\mathbb Z^3} a_{1,k_1}a_{2,k_2}a_{3,k_3}I_{\vec k},
	\]
	where, in polar coordinates $(y_1,y_2)=(r\theta_1,r\theta_2)$,
	\[
	I_{\vec k} = \int_{\mathbb S^1}\Omega(\theta)\int_{\mathbb R}\int_{1/2}^{2}
	\zeta_\theta(x,r)\,e^{-2\pi i P_{\vec k,\theta}(x,r)}\,dr\,dx\,d\sigma(\theta),
	\]
	with phase $P_{\vec k,\theta}(x,r)=\tfrac{1}{10}(k_1+k_2-k_3)x
	-\tfrac{r}{10}(k_1,k_2)\cdot\theta$ and amplitude
	$\zeta_\theta(x,r)=r^{-1}\beta(r)\widetilde\eta(x-r\theta_1)\widetilde\eta(x-r\theta_2)$.
	
	By Bessel's inequality, we establish the base $\ell^2(\mathbb Z)$ bounds
	\[
	\|a_{l, \cdot} \|_{\ell^2(\mathbb Z)} \lesssim \|f_{l,\lambda}\|_{L^2(\mathbb R)}, \qquad l=1,2,3.
	\]
	Substituting the localized functions yields
	\begin{align}
		\|a_{l, \cdot} \|_{\ell^2(\mathbb Z)} &\lesssim \|\widetilde\eta f_l\|_2, \qquad l=1,2, \label{eq-l2-12} \\
		\|a_{3, \cdot} \|_{\ell^2(\mathbb Z)} &\lesssim \|f_3\|_\infty. \label{eq-l2-3}
	\end{align}
	
	For $l=1,2$, we apply Lemma \ref{lemGN_bound} to bound the localized $L^\infty$ norms. For $l=3$, since $\|\phi_\lambda\|_{L^1} = \|\phi\|_{L^1} \lesssim 1$, Young's convolution inequality directly yields the bound without any $\lambda$ growth
	\begin{align*}
		\|f_{l,\lambda}\|_\infty &\lesssim \lambda^{1/2}\|\widetilde\eta f_l\|_2, \qquad l=1,2, \\
		\|f_{3,\lambda}\|_\infty &\le \|\eta\|_\infty \|f_3 * \phi_\lambda\|_\infty \lesssim \|\phi\|_{L^1} \|f_3\|_\infty \lesssim \|f_3\|_\infty.
	\end{align*}
	
	By \cite[Lemma~4.1]{bhojak2025alternate} applied to the $L^\infty$ bounds above:
	\begin{align}
		\|a_{l, \cdot} \|_{\ell^1(|k_l| > \lambda^{1+\varepsilon_1})} &\lesssim \lambda^{-100}\|f_{l,\lambda}\|_\infty \lesssim \lambda^{-99.5}\|\widetilde\eta f_l\|_2, \qquad l = 1, 2, \label{eq-high-l1} \\
		\|a_{3, \cdot} \|_{\ell^1(|k_3| > \lambda^{1+\varepsilon_1})} &\lesssim \lambda^{-100}\|f_{3,\lambda}\|_\infty \lesssim \lambda^{-100}\|f_3\|_\infty, \label{eq-high-l3} \\
		\|a_{l, \cdot} \|_{\ell^2(|k_l| \le \frac{3}{8}\lambda)} &\lesssim \lambda^{-100}\|f_{l,\lambda}\|_\infty \lesssim \lambda^{-99.5}\|\widetilde\eta f_l\|_2, \qquad l = 1, 2. \label{eq-low-l2}
	\end{align}
	
	By the Cauchy-Schwarz inequality, the $\ell^2$ low-frequency tail \eqref{eq-low-l2} converts to an $\ell^1$ tail
	\begin{equation}\label{eq-low-l1}
		\|a_{l, \cdot} \|_{\ell^1(|k_l| \le \frac{3}{8}\lambda)} \le \Big(\sum_{|k_l| \le \frac{3}{8}\lambda} 1\Big)^{1/2} \|a_{l, \cdot} \|_{\ell^2(|k_l| \le \frac{3}{8}\lambda)} \lesssim \lambda^{1/2} \big( \lambda^{-99.5}\|\widetilde\eta f_l\|_2 \big) = \lambda^{-99}\|\widetilde\eta f_l\|_2.
	\end{equation}
	
	Similarly, combining Cauchy-Schwarz on the main frequency region with the high-frequency $\ell^1$ tail \eqref{eq-high-l1}, we deduce the global $\ell^1(\mathbb Z)$ bounds
	\begin{align}
		\|a_{l, \cdot}\|_{\ell^1(\mathbb{Z})} &\le \Big(\sum_{|k_l| \le \lambda^{1+\varepsilon_1}} 1\Big)^{1/2} \|a_{l,\cdot}\|_{\ell^2(\mathbb{Z})} + \|a_{l, \cdot}\|_{\ell^1(|k_l| > \lambda^{1+\varepsilon_1})} \nonumber \\
		&\lesssim \lambda^{\frac{1+\varepsilon_1}{2}} \|\widetilde\eta f_l\|_2 + \lambda^{-99.5}\|\widetilde\eta f_l\|_2 \lesssim \lambda^{\frac{1+\varepsilon_1}{2}} \|\widetilde\eta f_l\|_2, \qquad l=1,2. \label{eq-global-l1-12}
	\end{align}
	By identical steps utilizing \eqref{eq-l2-3} and \eqref{eq-high-l3}, we have
	\begin{equation}\label{eq-global-l1-3}
		\|a_{3, \cdot}\|_{\ell^1(\mathbb{Z})} \lesssim \lambda^{\frac{1+\varepsilon_1}{2}} \|f_3\|_\infty.
	\end{equation}
	
	We now decompose the total sum into a main term and a tail error. Let $S_{\text{main}} = \{ \vec{k} \in \mathbb{Z}^3 : \frac{3}{8}\lambda \le |k_1|, |k_2| \le \lambda^{1+\varepsilon_1}, |k_3| \le \lambda^{1+\varepsilon_1} \}$, then
	\begin{equation}\label{eqmain-plus-errors}
		\Big| \sum_{\vec{k} \in \mathbb{Z}^3} a_{1,k_1}a_{2,k_2}a_{3,k_3}I_{\vec k} \Big| \le \Big| \sum_{\vec{k} \in S_{\text{main}}} a_{1,k_1}a_{2,k_2}a_{3,k_3}I_{\vec k} \Big| + E_{\text{tail}},
	\end{equation}
	where $|I_{\vec k}| \le \|\Omega\|_{L^1(\mathbb S^1)} \le \|\Omega\|_{\mathcal K_a}$. The condition $\vec{k} \notin S_{\text{main}}$ implies at least one index $k_l$ belongs to its rapid-decay tail region $\mathcal{T}_l$, defined as
	\begin{align*}
		\mathcal{T}_l &= \{k_l \in \mathbb{Z} : |k_l| > \lambda^{1+\varepsilon_1} \text{ or } |k_l| < \frac{3}{8}\lambda\}, \qquad l=1,2, \\
		\mathcal{T}_3 &= \{k_3 \in \mathbb{Z} : |k_3| > \lambda^{1+\varepsilon_1}\}.
	\end{align*}
	Combining \eqref{eq-high-l1}, \eqref{eq-high-l3}, and \eqref{eq-low-l1}, the total $\ell^1$ sum over these tail regions is bounded by
	\begin{align}
		\|a_{l, \cdot}\|_{\ell^1(\mathcal{T}_l)} &\lesssim \lambda^{-99}\|\widetilde\eta f_l\|_2, \qquad l=1,2, \label{eq-tail-total-12} \\
		\|a_{3, \cdot}\|_{\ell^1(\mathcal{T}_3)} &\lesssim \lambda^{-100}\|f_3\|_\infty. \label{eq-tail-total-3}
	\end{align}
	
	By the union bound, $E_{\text{tail}}$ is strictly controlled by summing the products where exactly one index is restricted to $\mathcal{T}_l$, and the other two run freely over $\mathbb{Z}$:
	\begin{align*}
		E_{\text{tail}} &\le \|\Omega\|_{\mathcal K_a} \sum_{\vec{k} \notin S_{\text{main}}} \big|a_{1,k_1}\big| \big|a_{2,k_2}\big| \big|a_{3,k_3}\big| \\
		&\lesssim \|\Omega\|_{\mathcal K_a} \Big( \|a_{1, \cdot}\|_{\ell^1(\mathcal{T}_1)} \|a_{2, \cdot}\|_{\ell^1(\mathbb{Z})} \|a_{3, \cdot}\|_{\ell^1(\mathbb{Z})} \\
		&\quad + \|a_{1, \cdot}\|_{\ell^1(\mathbb{Z})} \|a_{2, \cdot}\|_{\ell^1(\mathcal{T}_2)} \|a_{3, \cdot}\|_{\ell^1(\mathbb{Z})} \\
		&\quad + \|a_{1, \cdot}\|_{\ell^1(\mathbb{Z})} \|a_{2, \cdot}\|_{\ell^1(\mathbb{Z})} \|a_{3, \cdot}\|_{\ell^1(\mathcal{T}_3)} \Big).
	\end{align*}
	
	Substituting \eqref{eq-tail-total-12}--\eqref{eq-global-l1-3} into each of 
	the three terms and using $-99 + \frac{1+\varepsilon_1}{2} + \frac{1+\varepsilon_1}{2} 
	< -97$, all three terms are $O(\lambda^{-97}\|\widetilde\eta f_1\|_2
	\|\widetilde\eta f_2\|_2\|f_3\|_\infty)$ (and similarly for permutations). 
	Hence
	\begin{equation}\label{eqA-terms}
		E_{\text{tail}} \lesssim \lambda^{-97} \|\Omega\|_{\mathcal K_a} \|\widetilde\eta f_1\|_2\|\widetilde\eta f_2\|_2\|f_3\|_\infty.
	\end{equation}
	
	It remains to estimate the main sum over $S_{\text{main}}$. Following 
	\cite[Section~4.4]{bhojak2025alternate}, we fix $0<\varepsilon_2<1$ and 
	split $S_{\text{main}}$ according to
	\begin{align*}
		S_1 :& \quad |k_1+k_2-k_3| > \lambda^{\varepsilon_2}, \\
		S_2^+ :& \quad |k_1+k_2-k_3| \le \lambda^{\varepsilon_2},\quad 
		|(k_1,k_2)\cdot\theta| > \lambda^{\varepsilon_2}, \\
		S_2^- :& \quad |k_1+k_2-k_3| \le \lambda^{\varepsilon_2},\quad 
		|(k_1,k_2)\cdot\theta| \le \lambda^{\varepsilon_2}.
	\end{align*}
	
	\textit{Estimates for $S_1$ and $S_2^+$.}
	For $S_1$, the $x$-phase satisfies $|\partial_x P_{\vec k,\theta}|=
	|k_1+k_2-k_3|>\lambda^{\varepsilon_2}$; $N$-fold integration by parts 
	in $x$ gives $|I_{\vec k}|\lesssim\lambda^{-\varepsilon_2N}\|\Omega\|_{L^1(\mathbb S^1)}$.
	Summing over $S_{\text{main}}$ via Cauchy--Schwarz as in 
	\cite[Section~4.4]{bhojak2025alternate} and using 
	$\|\Omega\|_{L^1(\mathbb S^1)}\le\|\Omega\|_{\mathcal K_a}$ yields,
	for $N$ large enough,
	\[
	|S_1| \lesssim \lambda^{-50}\|\Omega\|_{\mathcal K_a}
	\|\widetilde\eta f_1\|_2\|\widetilde\eta f_2\|_2\|f_3\|_\infty.
	\]
	For $S_2^+$, the $r$-phase satisfies $|\partial_r P_{\vec k,\theta}|=
	|(k_1,k_2)\cdot\theta|>\lambda^{\varepsilon_2}$; the identical argument 
	(integration by parts in $r$) gives
	\[
	|S_2^+| \lesssim \lambda^{-50}\|\Omega\|_{\mathcal K_a}
	\|\widetilde\eta f_1\|_2\|\widetilde\eta f_2\|_2\|f_3\|_\infty.
	\]
	Hence
	\begin{equation}\label{eqS1S2plus}
		|S_1|+|S_2^+|\lesssim\lambda^{-50}\|\Omega\|_{\mathcal K_a}
		\|\widetilde\eta f_1\|_2\|\widetilde\eta f_2\|_2\|f_3\|_\infty.
	\end{equation}
	
	For the non-oscillatory term $S_2^-$, the summation is restricted to $|k_1+k_2-k_3| \le \lambda^{\varepsilon_2}$ and the angular phase condition $|(k_1,k_2)\cdot \theta| \le \lambda^{\varepsilon_2}$. For $\vec{k} \in S_{\text{main}}$, we have $|k_1| \ge \frac{3}{8}\lambda$, which strictly implies $|(k_1,k_2)| \ge \frac{3}{8}\lambda$. Defining the unit vector $\omega = \frac{(k_1,k_2)}{|(k_1,k_2)|} \in \mathbb{S}^1$, the phase condition simplifies geometrically
	\[
	|\omega \cdot \theta| = \frac{|(k_1,k_2)\cdot \theta|}{|(k_1,k_2)|} \le \frac{\lambda^{\varepsilon_2}}{3/8 \lambda} \le 4\lambda^{-1+\varepsilon_2}.
	\]
	
	By taking the absolute value inside the integral, $|I_{\vec k}|$ is bounded by the measure of this singular region. Factoring out the supremum over all directions $\omega \in \mathbb{S}^1$ yields
	\[
	|S_2^-| \lesssim \Bigg( \sup_{\omega\in\mathbb S^1} \int_{|\omega\cdot \theta|\le 4\lambda^{-1+\varepsilon_2}} |\Omega(\theta)|\,d\sigma(\theta) \Bigg) \sum_{\substack{\vec{k} \in S_{\text{main}} \\ |k_1+k_2-k_3| \le \lambda^{\varepsilon_2}}} |a_{1,k_1}| |a_{2,k_2}| |a_{3,k_3}|.
	\]
	
	To bound the restricted coefficient sum, we introduce an integer shift $\nu = k_3 - k_1 - k_2$ satisfying $|\nu| \le \lambda^{\varepsilon_2}$. Since there are $O(\lambda^{\varepsilon_2})$ such integers, a two-fold application of the Cauchy-Schwarz inequality (first in $k_2$, then in $k_1$) establishes the bound
	\begin{align}
		\sum_{\substack{\vec{k} \in S_{\text{main}} \\ |k_1+k_2-k_3| \le \lambda^{\varepsilon_2}}} \prod_{l=1}^3 |a_{l,k_l}| 
		&\le \sum_{|k_1| \le \lambda^{1+\varepsilon_1}} |a_{1,k_1}| \sum_{|\nu| \le \lambda^{\varepsilon_2}} \sum_{|k_2| \le \lambda^{1+\varepsilon_1}} |a_{2,k_2}| |a_{3,k_1+k_2+\nu}| \nonumber \\
		&\le \sum_{|k_1| \le \lambda^{1+\varepsilon_1}} |a_{1,k_1}| \sum_{|\nu| \le \lambda^{\varepsilon_2}} \|a_{2,\cdot}\|_{\ell^2} \|a_{3,\cdot}\|_{\ell^2} \nonumber \\
		&\lesssim \lambda^{\varepsilon_2} \|a_{2,\cdot}\|_{\ell^2} \|a_{3,\cdot}\|_{\ell^2} \sum_{|k_1| \le \lambda^{1+\varepsilon_1}} |a_{1,k_1}| \nonumber \\
		&\le \lambda^{\varepsilon_2} \|a_{2,\cdot}\|_{\ell^2} \|a_{3,\cdot}\|_{\ell^2} \Big( \sum_{|k_1| \le \lambda^{1+\varepsilon_1}} 1 \Big)^{1/2} \|a_{1,\cdot}\|_{\ell^2} \nonumber \\
		&\lesssim \lambda^{\varepsilon_2 + \frac{1+\varepsilon_1}{2}} \|a_{1,\cdot}\|_{\ell^2} \|a_{2,\cdot}\|_{\ell^2} \|a_{3,\cdot}\|_{\ell^2}. \label{eqS2minus-sum-bound}
	\end{align}
	
	Combining \eqref{eqS2minus-sum-bound} with the angular integral bound, we directly obtain
	\begin{equation}\label{eqS2minus-pre}
		|S_2^-| \lesssim \lambda^{\varepsilon_2+\frac{1+\varepsilon_1}{2}} \|a_{1,\cdot}\|_{\ell^2}\|a_{2,\cdot}\|_{\ell^2}\|a_{3,\cdot}\|_{\ell^2} \sup_{\omega\in\mathbb S^1} \int_{|\omega\cdot \theta|\le 4\lambda^{-1+\varepsilon_2}} |\Omega(\theta)|\,d\sigma(\theta).
	\end{equation}
	
	Now the new kernel condition \eqref{alpha} enters. For any $\omega\in\mathbb S^1$ and any $\delta>0$,
	\begin{equation}\label{eqsmallcap}
		\int_{|\omega\cdot\theta|\le \delta}|\Omega(\theta)|\,d\sigma(\theta) \le \delta^a \int_{\mathbb S^1}\frac{|\Omega(\theta)|}{|\omega\cdot\theta|^a}\,d\sigma(\theta) \le \delta^a \|\Omega\|_{\mathcal K_a}.
	\end{equation}
	Applying \eqref{eqsmallcap} with $\delta=4\lambda^{-1+\varepsilon_2}$ and combining with \eqref{eqS2minus-pre}, we conclude that
	\begin{equation}\label{eqS2minus-final}
		|S_2^-| \lesssim \|\Omega\|_{\mathcal K_a}\, \lambda^{\varepsilon_2+\frac{1+\varepsilon_1}{2}-a(1-\varepsilon_2)} \|\widetilde\eta f_1\|_2\|\widetilde\eta f_2\|_2\|f_3\|_\infty.
	\end{equation}
	To ensure a rapid polynomial decay, we require the exponent of $\lambda$ to be strictly negative, which is equivalent to
	\[
	\varepsilon_2 + \frac{1+\varepsilon_1}{2} - a(1-\varepsilon_2) < 0 \quad \iff \quad 2(1+a)\varepsilon_2 + \varepsilon_1 < 2a - 1.
	\]
	Since $a > 1/2$, the right-hand side $2a - 1$ is strictly positive. This guarantees the existence of sufficiently small $\varepsilon_1, \varepsilon_2 > 0$ satisfying the inequality. Denoting this negative exponent by $-c < 0$, we obtain
	\begin{equation}\label{eq:S2minus-decay}
		|S_2^-| \lesssim \lambda^{-c}\|\Omega\|_{\mathcal K_a}\|\widetilde\eta f_1\|_2\|\widetilde\eta f_2\|_2\|f_3\|_\infty.
	\end{equation}
	
	Finally, combining \eqref{eqA-terms}, \eqref{eqS1S2plus}, and \eqref{eq:S2minus-decay}, the proof is complete.
\end{proof}

With the single-scale decay estimates established, we are now ready to complete the proof of Theorem~\ref{thmmain}. To this end, we return to the decomposition~\eqref{decomp_main} and follow the multi-scale summation scheme of~\cite[Section~5]{bhojak2025alternate}, adapted to kernels satisfying condition~\eqref{alpha}.

\begin{proof}[Proof of Theorem~\ref{thmmain}]
	
	\medskip
	\noindent
	
\noindent\textbf{Step 1. Low-frequency term.}
	By \cite[Lemma~5.1 and Theorem~5.2]{bhojak2025alternate}, the multiplier 
	of $\mathcal{T}_\Omega^{LL}$ satisfies the Coifman--Meyer condition with 
	constant $\lesssim\|\Omega\|_{L^1(\mathbb S^1)}\le\|\Omega\|_{\mathcal K_a}$, 
	so
	\[
	\|\mathcal{T}_\Omega^{LL}(f_1,f_2)\|_{L^p}
	\lesssim
	\|\Omega\|_{\mathcal K_a}\|f_1\|_{L^{p_1}}\|f_2\|_{L^{p_2}}
	\]
	for all $1<p_1,p_2<\infty$, $\frac1p=\frac1{p_1}+\frac1{p_2}$, $p>\frac12$.
	
	\medskip
	\noindent
	\textbf{Step 2. Admissible growth bounds for the medium/high-frequency pieces.}  
	For the medium/high-frequency pieces \(T_{\Omega,k}^{HH}\), \(T_{\Omega,k}^{HL}\), and \(T_{\Omega,k}^{LH}\), we apply the same argument as in \cite[Lemma~5.3]{bhojak2025alternate}, replacing \(\|\Omega\|_1\) by \(\|\Omega\|_{\mathcal K_a}\). This yields the following growth bounds
	\begin{align}
		\|T_{\Omega,k}^{HH}(f_1,f_2)\|_{L^p}
		&\lesssim
		(1+k)^{\left|\frac1{p_1}-\frac12\right|+\left|\frac1{p_2}-\frac12\right|}
		\|\Omega\|_{\mathcal K_a}\,
		\|f_1\|_{L^{p_1}}\,
		\|f_2\|_{L^{p_2}},\label{eqHH-growth}\\
		\|T_{\Omega,k}^{HL}(f_1,f_2)\|_{L^p}
		&\lesssim
		(1+k)^{\left|\frac1{p_1}-\frac12\right|+\left|\frac1{p'}-\frac12\right|}
		\|\Omega\|_{\mathcal K_a}\,
		\|f_1\|_{L^{p_1}}\,
		\|f_2\|_{L^{p_2}},\label{eqHL-growth}\\
		\|T_{\Omega,k}^{LH}(f_1,f_2)\|_{L^p}
		&\lesssim
		(1+k)^{\left|\frac1{p'}-\frac12\right|+\left|\frac1{p_2}-\frac12\right|}
		\|\Omega\|_{\mathcal K_a}\,
		\|f_1\|_{L^{p_1}}\,
		\|f_2\|_{L^{p_2}}.\label{eqLH-growth}
	\end{align}
	
	\medskip
	\noindent
	\textbf{Step 3. Decay estimates from Proposition~\ref{prop_decay_strong}.}
	By Remark~\ref{rem_scale_inv}, Proposition~\ref{prop_decay_strong} applies 
	at each scale $i\in\mathbb Z$ with $\lambda\sim 2^k$. For $HH$, applying 
	the proposition and summing over $i$ via Cauchy--Schwarz and 
	Littlewood--Paley theory gives
	\begin{equation}\label{eqHH-decay-221}
		\|T_{\Omega,k}^{HH}(f_1,f_2)\|_{L^1}
		\lesssim 2^{-ck}\|\Omega\|_{\mathcal K_a}\|f_1\|_{L^2}\|f_2\|_{L^2}.
	\end{equation}
	The same argument applied to $HL$ and $LH$ yields
	\begin{align}
		\|T_{\Omega,k}^{HL}(f_1,f_2)\|_{L^2}
		&\lesssim 2^{-ck}\|\Omega\|_{\mathcal K_a}
		\|f_1\|_{L^2}\|f_2\|_{L^\infty},\label{eqHL-decay-2inf2}\\
		\|T_{\Omega,k}^{LH}(f_1,f_2)\|_{L^2}
		&\lesssim 2^{-ck}\|\Omega\|_{\mathcal K_a}
		\|f_1\|_{L^\infty}\|f_2\|_{L^2}.\label{eqLH-decay-inf22}
	\end{align}
	
	\medskip
	\noindent
	\textbf{Step 4. Interpolation for each fixed $k$.}
	For each family $I\in\{HL,LH,HH\}$ and each target $(1/p_1,1/p_2)$ in 
	the Banach triangle, we interpolate between the decay endpoint from 
	Step~3 and a suitable Banach growth point from Step~2, using bilinear 
	complex interpolation. The interpolation parameter $\theta\in(0,1)$ is 
	chosen small enough so that the growth point lies in the valid range and 
	the output exponent satisfies $1/p<1$; the polynomial growth factor 
	$(1+k)^{C(1-\theta)}$ is then absorbed into the exponential $2^{-c_0\theta k}$.
	This yields, for some $c_I>0$ depending only on $p_1,p_2$:
	
	\smallskip
	\noindent\textit{The $HL$ family.}
	\begin{equation}\label{eqHL-summable}
		\|T_{\Omega,k}^{HL}(f_1,f_2)\|_{L^p}
		\lesssim 2^{-c_{HL}k}\|\Omega\|_{\mathcal K_a}
		\|f_1\|_{L^{p_1}}\|f_2\|_{L^{p_2}}.
	\end{equation}
	
	\smallskip
	\noindent\textit{The $LH$ family.}
	\begin{equation}\label{eqLH-summable}
		\|T_{\Omega,k}^{LH}(f_1,f_2)\|_{L^p}
		\lesssim 2^{-c_{LH}k}\|\Omega\|_{\mathcal K_a}
		\|f_1\|_{L^{p_1}}\|f_2\|_{L^{p_2}}.
	\end{equation}
	
	\smallskip
	\noindent\textit{The $HH$ family.}
	\begin{equation}\label{eqHH-summable}
		\|T_{\Omega,k}^{HH}(f_1,f_2)\|_{L^p}
		\lesssim 2^{-c_{HH}k}\|\Omega\|_{\mathcal K_a}
		\|f_1\|_{L^{p_1}}\|f_2\|_{L^{p_2}}.
	\end{equation}
	
	Combining \eqref{eqHL-summable}, \eqref{eqLH-summable}, and \eqref{eqHH-summable}, and summing over \(k\), we obtain
	\[
	\|\mathcal {T}_\Omega(f_1,f_2)\|_{L^p}
	\lesssim
	\|\Omega\|_{\mathcal K_a}\,
	\|f_1\|_{L^{p_1}}\,
	\|f_2\|_{L^{p_2}}.
	\]
	This completes the proof of Theorem~\ref{thmmain}.
\end{proof}

\section{Maximal and maximally truncated operators}\label{secmaximal}

In this section, we prove the maximal and maximally truncated counterparts of Theorem~\ref{thmmain}. The ideas are essentially the same as in~\cite[Theorems~1.5 and~1.7, and Sections~7 and~8]{bhojak2025alternate}. The only new ingredient is that the kernel assumption is now condition~\eqref{alpha}, so that $\|\Omega\|_{L^1(\mathbb S^1)}$ is systematically replaced by $\|\Omega\|_{\mathcal K_a}$, while the required decay estimates are supplied by the results established in the previous section. For the reader's convenience, we include the main steps of the proof.

\begin{proof}[Proof of Theorem~\ref{thmmax}]
	We follow \cite[Section~7 and Theorem~1.7]{bhojak2025alternate} with
	$\|\Omega\|_{L^1(\mathbb S^1)}$ replaced throughout by
	$\|\Omega\|_{\mathcal K_a}$, which is valid since
	\[
	\|\Omega\|_{L^1(\mathbb S^1)}
	\le\sup_{\xi\in\mathbb S^1}\int_{\mathbb S^1}
	\frac{|\Omega(\theta)|}{|\theta\cdot\xi|^a}\,d\sigma(\theta)
	=\|\Omega\|_{\mathcal K_a}.
	\]
	
	\medskip
	\noindent
	\textbf{Step 1. Frequency decomposition.}
	As in \cite[Section~7]{bhojak2025alternate}, $M_\Omega$ is pointwise
	dominated by the dyadic annular supremum $\sup_{i\in\mathbb Z}
	\mathcal{T}_\Omega^i$. Applying the same Littlewood--Paley decomposition
	as in the proof of Theorem~\ref{thmmain} gives
	\[
	\sup_{i\in\mathbb Z}|\mathcal{T}_\Omega^i(f_1,f_2)(x)|
	\lesssim M_\Omega^{LL}(f_1,f_2)(x)
	+\sum_{k\ge 1}\bigl(M_{\Omega,k}^{HL}+M_{\Omega,k}^{LH}
	+M_{\Omega,k}^{HH}\bigr)(f_1,f_2)(x),
	\]
	where $M_\Omega^{LL}$, $M_{\Omega,k}^{HL}$, $M_{\Omega,k}^{LH}$,
	$M_{\Omega,k}^{HH}$ are defined as in
	\cite[Section~7]{bhojak2025alternate}.
	Setting $T_{\Omega,k}^{I,i}$ for the single-scale pieces as in
	\cite[(7.4)]{bhojak2025alternate}, we have the pointwise bounds
	\begin{equation}\label{eqmax-pointwise-1}
		|M_{\Omega,k}^I(f_1,f_2)(x)|
		\le\sum_{i\in\mathbb Z}|T_{\Omega,k}^{I,i}(f_1,f_2)(x)|,
		\qquad I\in\{HL,LH,HH\},
	\end{equation}
	\begin{equation}\label{eqmax-pointwise-2}
		|M_{\Omega,k}^I(f_1,f_2)(x)|
		\le\Bigl(\sum_{i\in\mathbb Z}
		|T_{\Omega,k}^{I,i}(f_1,f_2)(x)|^2\Bigr)^{1/2},
		\qquad I\in\{HL,LH\},
	\end{equation}
	where \eqref{eqmax-pointwise-1} follows from
	$\sup_i|a_i|\le\sum_i|a_i|$, and \eqref{eqmax-pointwise-2} from
	$\sup_i|a_i|\le(\sum_i|a_i|^2)^{1/2}$.
	
	\medskip
	\noindent
	\textbf{Step 2. Low-low term.}
	For each $i\in\mathbb Z$, the bound
	$|f_j*\phi_i(x-2^{-i}y_j)|\lesssim Mf_j(x)$ gives
	\[
	M_\Omega^{LL}(f_1,f_2)(x)
	\lesssim\|\Omega\|_{\mathcal K_a}\,Mf_1(x)\,Mf_2(x).
	\]
	By H\"older's inequality and the Hardy--Littlewood maximal theorem,
	\[
	\|M_\Omega^{LL}(f_1,f_2)\|_{L^p}
	\lesssim\|\Omega\|_{\mathcal K_a}
	\|f_1\|_{L^{p_1}}\|f_2\|_{L^{p_2}}.
	\]
	
	\medskip
	\noindent
	\textbf{Step 3. Growth bounds.}
	For each fixed $k\ge 1$ and $I\in\{HL,LH,HH\}$, the argument of
	\cite[Lemma~7.1(1)]{bhojak2025alternate}---using
	\eqref{eqmax-pointwise-1}--\eqref{eqmax-pointwise-2}, the
	Fefferman--Stein vector-valued maximal inequality, and
	Littlewood--Paley theory---yields
	\begin{equation}\label{eqM-growth}
		\|M_{\Omega,k}^I(f_1,f_2)\|_{L^p}
		\lesssim(1+k)^{C_I}\|\Omega\|_{\mathcal K_a}
		\|f_1\|_{L^{p_1}}\|f_2\|_{L^{p_2}},
	\end{equation}
	where $C_{HH}=|\frac1{p_1}-\frac12|+|\frac1{p_2}-\frac12|$,
	$C_{HL}=|\frac1{p_1}-\frac12|+|\frac1{p'}-\frac12|$,
	$C_{LH}=|\frac1{p'}-\frac12|+|\frac1{p_2}-\frac12|$.
	
	\medskip
	\noindent
	\textbf{Step 4. Decay estimates.}
	By \eqref{eqmax-pointwise-1}, \eqref{eqmax-pointwise-2}, and the
	single-scale decay bounds \eqref{eqHH-decay-221},
	\eqref{eqHL-decay-2inf2}, \eqref{eqLH-decay-inf22} from
	Proposition~\ref{prop_decay_strong} (via Remark~\ref{rem_scale_inv}),
	the same Cauchy--Schwarz and Littlewood--Paley summation as in
	Step~3 of the proof of Theorem~\ref{thmmain} gives
	\begin{align}
		\|M_{\Omega,k}^{HH}(f_1,f_2)\|_{L^1}
		&\lesssim 2^{-ck}\|\Omega\|_{\mathcal K_a}
		\|f_1\|_{L^2}\|f_2\|_{L^2},\label{eqM-HH-decay}\\
		\|M_{\Omega,k}^{HL}(f_1,f_2)\|_{L^2}
		&\lesssim 2^{-ck}\|\Omega\|_{\mathcal K_a}
		\|f_1\|_{L^2}\|f_2\|_{L^\infty},\label{eqM-HL-decay}\\
		\|M_{\Omega,k}^{LH}(f_1,f_2)\|_{L^2}
		&\lesssim 2^{-ck}\|\Omega\|_{\mathcal K_a}
		\|f_1\|_{L^\infty}\|f_2\|_{L^2}.\label{eqM-LH-decay}
	\end{align}
	
	\medskip
	\noindent
	\textbf{Step 5. Interpolation and summation.}
	Interpolating between \eqref{eqM-growth} and the corresponding decay
	estimate in \eqref{eqM-HH-decay}--\eqref{eqM-LH-decay} exactly as
	in Step~4 of the proof of Theorem~\ref{thmmain}, we obtain for some
	$c_*>0$ and every $I\in\{HH,HL,LH\}$,
	\begin{equation}\label{eqMk-final-decay}
		\|M_{\Omega,k}^I(f_1,f_2)\|_{L^p}
		\lesssim 2^{-c_*k}\|\Omega\|_{\mathcal K_a}
		\|f_1\|_{L^{p_1}}\|f_2\|_{L^{p_2}}.
	\end{equation}
	Combining Step~2 with \eqref{eqMk-final-decay} and summing over
	$k\ge 1$ completes the proof.
\end{proof}

With the maximal estimate now established, we proceed to the maximally truncated operator. 

\begin{proof}[Proof of Theorem~\ref{thmtrunc}]
	We follow \cite[Section~8 and Theorem~1.5]{bhojak2025alternate}, again
	replacing $\|\Omega\|_{L^1(\mathbb S^1)}$ by $\|\Omega\|_{\mathcal K_a}$
	throughout.
	
	\medskip
	\noindent
	\textbf{Step 1. Reduction.}
	As in \cite[Section~8]{bhojak2025alternate}, the pointwise inequality
	\[
	\mathcal{T}_\Omega^*(f_1,f_2)(x)
	\le\sup_{j\in\mathbb Z}
	\Big|\sum_{i>j}\mathcal{T}_\Omega^i(f_1,f_2)(x)\Big|
	+M_{|\Omega|}(f_1,f_2)(x)
	\]
	holds. Since $\||\Omega|\|_{\mathcal K_a}=\|\Omega\|_{\mathcal K_a}$,
	the term $M_{|\Omega|}$ is controlled by Theorem~\ref{thmmax}:
	\[
	\|M_{|\Omega|}(f_1,f_2)\|_{L^p}
	\lesssim\|\Omega\|_{\mathcal K_a}
	\|f_1\|_{L^{p_1}}\|f_2\|_{L^{p_2}}.
	\]
	It therefore remains to estimate
	$\sup_{j\in\mathbb Z}|\sum_{i>j}\mathcal{T}_\Omega^i(f_1,f_2)|$.
	
	\medskip
	\noindent
	\textbf{Step 2. Frequency decomposition of the tail.}
	The same Littlewood--Paley decomposition as in the proof of
	Theorem~\ref{thmmax} gives
	\[
	\sup_{j\in\mathbb Z}
	\Big|\sum_{i>j}\mathcal{T}_\Omega^i(f_1,f_2)\Big|
	\lesssim T_\Omega^{LL,*}(f_1,f_2)
	+\sum_{k\ge 1}\bigl(T_{\Omega,k}^{HL,*}
	+T_{\Omega,k}^{LH,*}+T_{\Omega,k}^{HH,*}\bigr)(f_1,f_2),
	\]
	where $T_{\Omega,k}^{I,*}$ are defined as in
	\cite[Section~8]{bhojak2025alternate}.
	
	\medskip
	\noindent
	\textbf{Step 3. Low-low tail term.}
	The kernel $K^{LL}=\sum_i K_i*(\phi_i\otimes\phi_i)$ is a standard
	bilinear Calder\'on--Zygmund kernel, so Cotlar's inequality gives
	\[
	\|T_\Omega^{LL,*}(f_1,f_2)\|_{L^p}
	\lesssim\|\Omega\|_{L^1(\mathbb S^1)}
	\|f_1\|_{L^{p_1}}\|f_2\|_{L^{p_2}}
	\le\|\Omega\|_{\mathcal K_a}
	\|f_1\|_{L^{p_1}}\|f_2\|_{L^{p_2}}.
	\]
	
	\medskip
	\noindent
	\textbf{Step 4. Pointwise reductions for medium/high tail pieces.}
	For $I\in\{HL,LH,HH\}$, we have the straightforward bound
	\begin{equation}\label{eqtail-pointwise-1}
		|T_{\Omega,k}^{I,*}(f_1,f_2)(x)|
		\le\sum_{i\in\mathbb Z}|T_{\Omega,k}^{I,i}(f_1,f_2)(x)|.
	\end{equation}
	For $I\in\{HL,LH\}$, the Fourier support argument of
	\cite[pp.~22--23]{bhojak2025alternate} gives the refined estimate
	\begin{equation}\label{eqtail-pointwise-2}
		|T_{\Omega,k}^{I,*}(f_1,f_2)(x)|
		\lesssim|T_{\Omega,k}^I(f_1,f_2)(x)|
		+M_{HL}\bigl(T_{\Omega,k}^I(f_1,f_2)\bigr)(x)
		+M_{HL}\bigl(M_{\Omega,k}^I(f_1,f_2)\bigr)(x),
	\end{equation}
	where $M_{HL}$ is the Hardy--Littlewood maximal operator.
	
	\medskip
	\noindent
	\textbf{Step 5. Growth bounds.}
	For each fixed $k\ge 1$ and $I\in\{HL,LH,HH\}$:
	
	For $I=HH$, \eqref{eqtail-pointwise-1} and the same argument as
	in Step~3 of the proof of Theorem~\ref{thmmax} give
	\[
	\|T_{\Omega,k}^{HH,*}(f_1,f_2)\|_{L^p}
	\lesssim(1+k)^{C_{HH}}\|\Omega\|_{\mathcal K_a}
	\|f_1\|_{L^{p_1}}\|f_2\|_{L^{p_2}}.
	\]
	
	For $I\in\{HL,LH\}$, applying \eqref{eqtail-pointwise-2}, the
	$L^p$-boundedness of $M_{HL}$, the growth bounds
	\eqref{eqHL-growth}--\eqref{eqLH-growth}, and the maximal growth
	bound \eqref{eqM-growth} gives
	\[
	\|T_{\Omega,k}^{I,*}(f_1,f_2)\|_{L^p}
	\lesssim\|T_{\Omega,k}^I(f_1,f_2)\|_{L^p}
	+\|M_{\Omega,k}^I(f_1,f_2)\|_{L^p}
	\lesssim(1+k)^{C_I}\|\Omega\|_{\mathcal K_a}
	\|f_1\|_{L^{p_1}}\|f_2\|_{L^{p_2}}.
	\]
	
	Hence for all three families,
	\begin{equation}\label{eqtail-growth}
		\|T_{\Omega,k}^{I,*}(f_1,f_2)\|_{L^p}
		\lesssim(1+k)^{C_I}\|\Omega\|_{\mathcal K_a}
		\|f_1\|_{L^{p_1}}\|f_2\|_{L^{p_2}}.
	\end{equation}
	
	\medskip
	\noindent
	\textbf{Step 6. Decay bounds.}
	At the three anchor exponent triples, the following decay bounds hold.
	For $I=HH$, \eqref{eqtail-pointwise-1} and \eqref{eqM-HH-decay} give
	\begin{equation}\label{eqtail-decay-HH}
		\|T_{\Omega,k}^{HH,*}(f_1,f_2)\|_{L^1}
		\lesssim 2^{-ck}\|\Omega\|_{\mathcal K_a}
		\|f_1\|_{L^2}\|f_2\|_{L^2}.
	\end{equation}
	For $I=HL$, \eqref{eqtail-pointwise-2}, the $L^2$-boundedness of
	$M_{HL}$, \eqref{eqHL-decay-2inf2}, and \eqref{eqM-HL-decay} give
	\begin{equation}\label{eqtail-decay-HL}
		\|T_{\Omega,k}^{HL,*}(f_1,f_2)\|_{L^2}
		\lesssim\|T_{\Omega,k}^{HL}(f_1,f_2)\|_{L^2}
		+\|M_{\Omega,k}^{HL}(f_1,f_2)\|_{L^2}
		\lesssim 2^{-ck}\|\Omega\|_{\mathcal K_a}
		\|f_1\|_{L^2}\|f_2\|_{L^\infty}.
	\end{equation}
	By symmetry,
	\begin{equation}\label{eqtail-decay-LH}
		\|T_{\Omega,k}^{LH,*}(f_1,f_2)\|_{L^2}
		\lesssim 2^{-ck}\|\Omega\|_{\mathcal K_a}
		\|f_1\|_{L^\infty}\|f_2\|_{L^2}.
	\end{equation}
	
	\medskip
	\noindent
	\textbf{Step 7. Interpolation and summation.}
	Interpolating between \eqref{eqtail-growth} and the corresponding
	decay estimate \eqref{eqtail-decay-HH}--\eqref{eqtail-decay-LH}
	exactly as in Step~4 of the proof of Theorem~\ref{thmmain}, we obtain
	for some $c_{**}>0$ and every $I\in\{HH,HL,LH\}$,
	\begin{equation}\label{eqtail-final-decay}
		\|T_{\Omega,k}^{I,*}(f_1,f_2)\|_{L^p}
		\lesssim 2^{-c_{**}k}\|\Omega\|_{\mathcal K_a}
		\|f_1\|_{L^{p_1}}\|f_2\|_{L^{p_2}}.
	\end{equation}
	Combining the low-low bound from Step~3 and the $LL,*$ term from Step~2
	with \eqref{eqtail-final-decay} and summing over $k\ge 1$ completes the proof.
\end{proof}

\section{Relationship with the Orlicz Space}\label{sec4}

This section is devoted to the proof of Theorem~\ref{thmincomparability}. condition~\eqref{alpha} gives rise to a geometrically defined class of rough singularities, whose relation to the classical Lebesgue scale was already observed in~\cite{Duoandikoetxea}. More precisely, if $q>\frac{1}{1-a}$, then H\"older's inequality gives
\[
\int_{\mathbb{S}^{2n-1}}\frac{|\Omega(\theta)|}{|\theta\cdot\xi'|^a}\,d\sigma(\theta)
\leq
\|\Omega\|_{L^q(\mathbb{S}^{1})}
\Bigl(\int_{\mathbb{S}^{1}}|\theta\cdot\xi'|^{-aq'}\,d\sigma(\theta)\Bigr)^{1/q'}
\lesssim
\|\Omega\|_{L^q(\mathbb{S}^{1})}.
\]
Thus, one has the continuous inclusion $L^q(\mathbb{S}^{1})\subset \mathcal K_a$ whenever $q>\frac{1}{1-a}$. However, this observation lies far from the critical integrability regime relevant to rough singular integrals, where the natural near-$L^1$ substitutes are the Orlicz spaces $L(\log L)^\alpha$. The main purpose of this section is to show that condition~\eqref{alpha} is genuinely different from that Orlicz scale because neither condition dominates the other in general. In this sense, the gap between the fractional geometric condition and the classical logarithmic integrability assumptions is not merely technical, but reflects a genuine structural distinction.

The proof of Theorem~\ref{thmincomparability} is divided into two propositions. Part (i) is proved in Proposition~\ref{propforward_1d_orlicz}, and part (ii) in Proposition~\ref{propreverse_orlicz}.

\subsection{Proof of Theorem \ref{thmincomparability} (i): Failure of Orlicz and Lebesgue integrability}
The proof proceeds by constructing $\Omega$ as a sum of characteristic functions of carefully positioned tiny intervals, then verifying the fractional condition via a delicate geometric summation and the failure of the Orlicz condition through a level set argument.

\begin{proposition}\label{propforward_1d_orlicz}
	Let $0<a<1$ and $\alpha>1$. There exists a nonnegative function
	$\Omega \in L^1(\mathbb{S}^1)$ such that
	\[
	\Omega\in \mathcal K_a(\mathbb{S}^1)
	\]
	but
	\[
	\Omega\notin L(\log L)^\alpha(\mathbb{S}^1).
	\]
\end{proposition}

\begin{proof}
	We identify $\mathbb{S}^1$ with $[0,2\pi)$ endowed with periodic arc-length measure. For each integer $k \ge 1$, we first partition $\mathbb{S}^1$ into $N_k$ consecutive arcs of equal length $2\pi/N_k$, where the number of partition intervals is given by
	$$N_k = \Bigl\lceil k^{-\alpha}4^{k(\frac{1}{1-a} - 1)}\,k^{\frac{2}{1-a}}\Bigr\rceil.$$
	Due to the dominant exponential growth of $4^{k(\frac{1}{1-a} - 1)}$ (since $0<a<1$), we have $N_k \to \infty$ as $k \to \infty$. The ceiling function naturally guarantees that $N_k \ge 1$. 
	
	In the center of each such arc, we choose a closed subinterval $I_{k,j}$ for $1 \le j \le N_k$. 
	
	Next, we define the height and total mass parameters at level $k$ as
	$$h_k = 4^k, \qquad w_k = k^{-\alpha}4^{-k}.$$
	To uniformly distribute the mass $w_k$ around the circle, we set the length of each individual subinterval to be
	$$\delta_k = \frac{w_k}{N_k},$$
	so that $|I_{k,j}| = \delta_k$. Since $w_k \le 1$ for all $k \ge 1$, it is easy to observe that
	$$\delta_k = \frac{w_k}{N_k} \le \frac{1}{N_k} < \frac{2\pi}{N_k}.$$
	This strict inequality guarantees that the subintervals $I_{k,j}$ are strictly contained within their respective ambient arcs and are therefore pairwise disjoint within the fixed level $k$.
	
	We then define the set $E_k$ as the union of these uniformly distributed intervals
	$$E_k = \bigcup_{j=1}^{N_k} I_{k,j}.$$
	By construction, the measure of this set is precisely $\sigma(E_k) = N_k \delta_k = w_k$.
	
	Finally, having constructed the sets $E_k$ for all $k \ge 1$, we define the non-negative measurable function $\Omega$ on $\mathbb{S}^1$ by
	$$\Omega(\theta) = \sum_{k=1}^\infty h_k \chi_{E_k}(\theta).$$
	We now verify that this $\Omega$ satisfies all the required properties.
	
	\medskip
	\noindent\textbf{Step 1 \(\Omega\in L^1(\mathbb{S}^1)\).}
	
	Since $\Omega\ge 0$, Tonelli's theorem gives
	\[
	\int_{\mathbb{S}^1}\Omega(\theta)\,d\sigma(\theta)
	=\sum_{k=1}^\infty h_k \sigma(E_k)
	=\sum_{k=1}^\infty h_k w_k
	=\sum_{k=1}^\infty 4^k\cdot k^{-\alpha}4^{-k}
	=\sum_{k=1}^\infty k^{-\alpha}<\infty,
	\]
	because $\alpha>1$. Hence $\Omega\in L^1(\mathbb{S}^1)$.
	
	\medskip
	\noindent\textbf{Step 2 \(\Omega\notin L(\log L)^\alpha(\mathbb{S}^1)\).}
	
	Define
	\[
	A_k=E_k\setminus \bigcup_{m>k}E_m.
	\]
	Then
	\[
	\sigma(A_k)\ge \sigma(E_k)-\sum_{m>k}\sigma(E_m)
	=w_k-\sum_{m>k}w_m.
	\]
	Since $w_m=m^{-\alpha}4^{-m}$ and $m\mapsto m^{-\alpha}$ is decreasing,
	\[
	\sum_{m>k}w_m
	\le \frac{1}{(k+1)^\alpha}\sum_{m=k+1}^\infty 4^{-m}
	=\frac{1}{(k+1)^\alpha}\frac{4^{-k-1}}{1-1/4}
	=\frac{1}{3}(k+1)^{-\alpha}4^{-k}
	\le \frac{1}{3}k^{-\alpha}4^{-k}
	=\frac{1}{3}w_k.
	\]
	Therefore
	\[
	\sigma(A_k)\ge \frac{2}{3}w_k.
	\]
	If $\theta\in A_k$, then $\theta\in E_k$ and $\theta\notin E_m$ for all $m>k$, hence
	\[
	\Omega(\theta)=\sum_{m=1}^\infty h_m\chi_{E_m}(\theta)\ge h_k.
	\]
	It follows that
	\begin{align*}
		\int_{\mathbb{S}^1}\Omega(\theta)\bigl(\log(e+\Omega(\theta))\bigr)^\alpha\,d\sigma(\theta)
		&\ge \sum_{k=1}^\infty
		\int_{A_k} h_k\bigl(\log(e+h_k)\bigr)^\alpha\,d\sigma(\theta) \\
		&= \sum_{k=1}^\infty h_k\bigl(\log(e+h_k)\bigr)^\alpha \sigma(A_k).
	\end{align*}
	Using $h_k=4^k$, $\sigma(A_k)\ge \frac23 w_k$, and
	$\log(e+4^k)\ge k\log 4$, we obtain
	\begin{align*}
		\int_{\mathbb{S}^1}\Omega(\theta)\bigl(\log(e+\Omega(\theta))\bigr)^\alpha\,d\sigma(\theta)
		&\ge \frac{2}{3}\sum_{k=1}^\infty
		4^k (k\log 4)^\alpha w_k \\
		&= \frac{2}{3}(\log 4)^\alpha
		\sum_{k=1}^\infty 4^k k^\alpha \cdot k^{-\alpha}4^{-k} \\
		&= \frac{2}{3}(\log 4)^\alpha \sum_{k=1}^\infty 1
		= \infty.
	\end{align*}
	Therefore $\Omega\notin L(\log L)^\alpha(\mathbb{S}^1)$.
	
	\medskip
	\noindent\textbf{Step 3 \(\Omega\in \mathcal K_a(\mathbb{S}^1)\).}
	
	For each $k\ge 1$, define
	$$ I_k(\xi) = \int_{E_k} \frac{1}{|\theta \cdot \xi|^a} \, d\sigma(\theta). $$
	Then, by Tonelli's theorem,
	$$ \int_{\mathbb{S}^1} \frac{\Omega(\theta)}{|\theta \cdot \xi|^a} \, d\sigma(\theta) = \sum_{k=1}^\infty h_k I_k(\xi). $$
	Hence it suffices to prove that
	$$ \sup_{\xi\in\mathbb{S}^1} I_k(\xi) \lesssim_a \delta_k^{\,1-a} + w_k \qquad \text{uniformly in } k. $$
	
	Fix $\xi \in \mathbb{S}^1$ and define its orthogonal poles $P_\xi = \{\xi^\perp, -\xi^\perp\}$. For any $\theta \in \mathbb{S}^1$, the inner product satisfies
	$$ |\theta \cdot \xi| = \sin(d_{\mathbb{S}^1}(\theta, P_\xi)), \quad \text{where } d_{\mathbb{S}^1}(\theta, P_\xi) = \min_{p \in P_\xi} \arccos(\theta \cdot p). $$
	
	Recall that $\mathbb{S}^1$ is partitioned into $N_k$ arcs, with each subinterval $I_{k,j}$ centered in its ambient arc. We define the set of \emph{bad} intervals as
	$$ \mathcal{B}_k(\xi) = \{ I_{k,j}:I_{k,j} \cap P_\xi \neq \emptyset \}, $$
	and call the remaining intervals \emph{good}. Since there are exactly two poles in $P_\xi$, and each pole can intersect at most two adjacent partition arcs, there are at most $4$ bad intervals in $\mathcal{B}_k(\xi)$.
	
	For any bad interval $I_{k,j} \in \mathcal{B}_k(\xi)$, its ambient arc intersects a pole $p \in P_\xi$, introducing a singularity. Parametrizing $\theta \in I_{k,j}$ by its distance to the nearest pole, $\varphi = d_{\mathbb{S}^1}(\theta, P_\xi)$, gives $|\theta \cdot \xi| = \sin \varphi$. This maps $I_{k,j}$ to an angular interval $\widetilde{I}_{k,j}$. On the unit circle, the arc length equals the magnitude of its corresponding central angle, so $|\widetilde{I}_{k,j}| = |I_{k,j}| = \delta_k$ and $d\sigma(\theta) = d\varphi$. Thus, changing variables yields
	$$ \int_{I_{k,j}} |\theta \cdot \xi|^{-a} \, d\sigma(\theta) = \int_{\widetilde{I}_{k,j}} |\sin \varphi|^{-a} \, d\varphi. $$
	
	Observe that $p \in I_{k,j}$ is equivalent to $0 \in \widetilde{I}_{k,j}$ in the parameter domain. Because the value of $|\sin \varphi|^{-a}$ gets larger as it gets closer to $0$, the integral $\int_{\widetilde{I}_{k,j}} |\sin \varphi|^{-a} \, d\varphi$ is maximized when $0$ is positioned exactly at the center of $\widetilde{I}_{k,j}$. Therefore, we have
	$$\int_{\widetilde{I}_{k,j}} |\sin \varphi|^{-a} \, d\varphi \le \int_{-\delta_k/2}^{\delta_k/2} |\sin \varphi|^{-a} \, d\varphi.$$
	
	Since $\delta_k/2 \le \pi/2$, we have $|\sin \varphi| \ge \frac{2}{\pi} |\varphi|$, which yields
	$$ \int_{-\delta_k/2}^{\delta_k/2} |\sin \varphi|^{-a} \, d\varphi \le 2 \left(\frac{\pi}{2}\right)^a \int_0^{\delta_k/2} \varphi^{-a} \, d\varphi \lesssim_a \delta_k^{\,1-a}. $$
	Summing over the bounded number of bad intervals gives
	$$ \sum_{I_{k,j} \, \text{bad}} \int_{I_{k,j}} |\theta \cdot \xi|^{-a} \, d\sigma(\theta) \lesssim_a \delta_k^{\,1-a}.$$
	
	We call the remaining intervals \emph{good}. For $1 \le m \le \lfloor N_k/4 \rfloor$, we group these good intervals by their distance to the poles $P_\xi$ by defining
	$$ \mathcal{G}_{k,m}(\xi) = \left\{ I_{k,j} \notin \text{bad}  \frac{2\pi m}{N_k} \le d_{\mathbb{S}^1}(I_{k,j}, P_\xi) < \frac{2\pi (m+1)}{N_k} \right\}. $$
	By construction, each $\mathcal{G}_{k,m}$ contains at most 4 intervals $I_{k,j}$, as each of the two poles admits at most one corresponding interval in both the clockwise and counterclockwise directions at a given grid distance $m$. Furthermore, the distance from any point in these intervals to the nearest singularity is strictly bounded from below by
	$$ \inf_{\theta \in I_{k,j} \in \mathcal{G}_{k,m}(\xi)} d_{\mathbb{S}^1}(\theta, P_\xi) \ge \frac{2\pi m}{N_k} \gtrsim \frac{m}{N_k}. $$
	
	Consequently, for any interval $I_{k,j} \in \mathcal{G}_{k,m}(\xi)$, we can uniformly bound the singular kernel by
	$$ |\theta \cdot \xi|^{-a} = |\sin(d_{\mathbb{S}^1}(\theta, P_\xi))|^{-a} \lesssim \left(\frac{N_k}{m}\right)^a.$$
	Integrating this bound over the length $\delta_k$ of the interval gives
	$$ \int_{I_{k,j}} |\theta \cdot \xi|^{-a} \, d\sigma(\theta) \lesssim \delta_k \left(\frac{N_k}{m}\right)^a. $$
	Summing over the at most 4 intervals in $\mathcal{G}_{k,m}(\xi)$ and then over all possible discrete distances $m$ yields
	$$ \sum_{I_{k,j} \, \text{good}} \int_{I_{k,j}} |\theta \cdot \xi|^{-a} \, d\sigma(\theta) \lesssim \delta_k N_k^a \sum_{m=1}^{\lfloor N_k/4 \rfloor} m^{-a}. $$
	Since $0 < a < 1$, the sum is asymptotically comparable to $N_k^{1-a}$, and since $\delta_k = w_k/N_k$, we have
	$$ \sum_{I_{k,j} \, \text{good}} \int_{I_{k,j}} |\theta \cdot \xi|^{-a} \, d\sigma(\theta) \lesssim_a \delta_k N_k^a \cdot N_k^{1-a} = \delta_k N_k = w_k. $$
	
	Summing the contributions from both the bad and good intervals gives
	$$ I_k(\xi) = \int_{E_k} \frac{1}{|\theta \cdot \xi|^a} \, d\sigma(\theta) \lesssim_a \delta_k^{1-a} + w_k. $$
	Crucially, this uniform bound holds independently of the choice of $\xi \in \mathbb{S}^1$.
	
	Finally, taking the supremum over $\xi$ and summing over all $k$, we conclude
	$$ \sup_{\xi \in \mathbb{S}^1} \int_{\mathbb{S}^1} \frac{|\Omega(\theta)|}{|\theta \cdot \xi|^a} \, d\sigma(\theta) \lesssim_a \sum_{k=1}^\infty h_k \delta_k^{1-a} + \sum_{k=1}^\infty h_k w_k. $$
	The second sum converges because $\sum_{k=1}^\infty h_k w_k = \sum_{k=1}^\infty k^{-\alpha} < \infty$. For the first sum, since $N_k \ge w_k 4^{\frac{k}{1-a}} k^{\frac{2}{1-a}}$, we have $\delta_k = \frac{w_k}{N_k} \le 4^{-\frac{k}{1-a}} k^{-\frac{2}{1-a}}$, which implies $\delta_k^{1-a} \le 4^{-k} k^{-2}$. Thus,
	$$ \sum_{k=1}^\infty h_k \delta_k^{1-a} \le \sum_{k=1}^\infty 4^k \cdot 4^{-k} k^{-2} = \sum_{k=1}^\infty k^{-2} < \infty. $$
	Hence, we conclude that $\sup_{\xi \in \mathbb{S}^1} \int_{\mathbb{S}^1} \frac{|\Omega(\theta)|}{|\theta \cdot \xi|^a} \, d\sigma(\theta) < \infty$, completing the proof of Proposition~\ref{propforward_1d_orlicz}.
\end{proof}

\subsection{Proof of Theorem \ref{thmincomparability} (ii): Failure of the fractional condition}

Having established that Condition (\ref{alpha}) encompasses highly singular functions outside of the Orlicz space $L(\log L)^\alpha(\mathbb{S}^{1})$, we now prove part (ii), demonstrating the strict failure of the converse inclusion.

\begin{proposition}\label{propreverse_orlicz}
	Let $0<a<1$ and $\alpha>1$. There exists a nonnegative function
	$\Omega \in L^1(\mathbb{S}^1)$ such that
	\[
	\Omega\in L(\log L)^\alpha(\mathbb{S}^1)
	\]
	but
	\[
	\Omega\notin \mathcal K_a(\mathbb{S}^1).
	\]
\end{proposition}

\begin{proof}
	Fix a pole $e_1 \in \mathbb{S}^1$. We define the non-negative function $\Omega: \mathbb{S}^1 \to \mathbb{R}$ explicitly by$$\Omega(\theta) = \frac{1}{|\theta \cdot e_1|^{1-a}}.$$
	
	\medskip
	\noindent\textbf{Step 1 $\Omega\in L(\log L)^\alpha(\mathbb{S}^1)$.}
	
	We will show that $\Omega$ belongs to a standard Lebesgue space $L^q(\mathbb{S}^1)$ for some $q > 1$, which continuously embeds into the Orlicz space $L(\log L)^\alpha(\mathbb{S}^1)$ since the sphere has a finite measure.
	Choose a fixed exponent $q$ such that $1 < q < \frac{1}{1-a}$.
	
	Parametrizing the circle $\mathbb{S}^1$ by an angle $\varphi \in [-\pi, \pi)$ such that $\theta \cdot e_1 = \sin\varphi$, the integration over the circle reduces to a simple one-dimensional trigonometric integral. Exploiting the symmetries of the sine function, we can restrict the integration domain to $[0, \pi/2]$ and multiply the result by a factor of $4$, yielding
	$$\|\Omega\|_{L^q}^q = \int_{\mathbb{S}^1} |\theta \cdot e_1|^{-q(1-a)} \, d\sigma(\theta) = \int_{-\pi}^\pi |\sin \varphi|^{-q(1-a)} \, d\varphi = 4 \int_0^{\pi/2} (\sin \varphi)^{-q(1-a)} \, d\varphi.$$
	
	On the interval $[0, \pi/2]$, we can use the standard bound $\sin \varphi \ge \frac{2}{\pi} \varphi$ to estimate the integral from above
	$$\|\Omega\|_{L^q}^q \le 4 \left(\frac{\pi}{2}\right)^{q(1-a)} \int_0^{\pi/2} \varphi^{-q(1-a)} \, d\varphi.$$
	
	By our hypothesis, $1 < q < \frac{1}{1-a}$, which rearranges to $q(1-a) < 1$.
	Consequently, the power of $\varphi$ in the integrand is strictly greater than $-1$, and the integral converges to a finite constant. This rigorously proves that $\Omega \in L^q(\mathbb{S}^1)$.
	Because $\mathbb{S}^1$ is a space of finite measure, the continuous embedding $L^q(\mathbb{S}^1) \subset L(\log L)^\alpha(\mathbb{S}^1)$ holds for any $q > 1$ and $\alpha > 0$. Therefore, $\Omega \in L(\log L)^\alpha(\mathbb{S}^1)$.
	
	\noindent\textbf{Step 2 $\Omega\notin \mathcal K_a(\mathbb{S}^1)$.}
	
	Evaluating the integral at the specific pole $\xi = e_1$ yields a strict lower bound for the supremum
	$$\sup_{\xi \in \mathbb{S}^1} \int_{\mathbb{S}^1} \frac{|\Omega(\theta)|}{|\theta \cdot \xi|^a} \, d\sigma(\theta) \ge \int_{\mathbb{S}^1} \frac{|\theta \cdot e_1|^{-(1-a)}}{|\theta \cdot e_1|^a} \, d\sigma(\theta) = \int_{\mathbb{S}^1} \frac{1}{|\theta \cdot e_1|} \, d\sigma(\theta).$$
	
	Applying the angular parametrization again gives
	$$\int_{\mathbb{S}^1} \frac{1}{|\theta \cdot e_1|} \, d\sigma(\theta) = 4 \int_0^{\pi/2} \frac{1}{\sin \varphi} \, d\varphi.$$
	
	To establish the divergence of the integral, we can bound the sine function from above by its argument, namely $\sin \varphi \le \varphi$ for $\varphi \in (0, \pi/2]$. This yields the following strict chain of inequalities
	$$4 \int_0^{\pi/2} \frac{1}{\sin \varphi} \, d\varphi \ge 4 \int_0^{\pi/2} \frac{1}{\varphi} \, d\varphi.$$
	
	Since the integral $\int_0^{\pi/2} \frac{1}{\varphi} \, d\varphi$ diverges to $+\infty$, the entire expression diverges. This completes the proof that $\Omega$ strictly fails the fractional condition.	
\end{proof}

\begin{remark}\label{rem_higher_dim}
	The strict non-inclusion established in Theorem~\ref{thmincomparability} naturally extends to $\mathbb{S}^{n-1}$ for $n \ge 3$.
	
	For part (i), the construction adapts to $\mathbb{S}^{n-1}$ by partitioning the sphere into a grid of $N_k$ fine cells. For any fixed pole $\xi \in \mathbb{S}^{n-1}$, the singularity $|\theta \cdot \xi|^{-a}$ is distributed along an $(n-2)$-dimensional equator. As the grid is refined, the number of cells intersecting this singular equator naturally increases. However, the geometric constraint $a < 1$ acts as a critical threshold: it guarantees that the shrinking volume of the cells strictly overpowers the $(n-2)$-dimensional growth of the singular set. Consequently, the total integral over the equator decays to zero as $N_k \to \infty$, allowing one to force absolute convergence simply by selecting a rapidly growing grid cardinality $N_k$.
	
	For part (ii), setting $\Omega(\theta) = |\theta \cdot e_1|^{-\beta}$ guarantees $\Omega \in L^q(\mathbb{S}^{n-1}) \subset L(\log L)^\alpha(\mathbb{S}^{n-1})$ for any $\beta < 1$. However, testing condition \eqref{alpha} at $\xi = e_1$ yields $\int_{\mathbb{S}^{n-1}} |\theta \cdot e_1|^{-(\beta+a)} \, d\sigma(\theta) = \infty$ if $\beta + a \ge 1$. Thus, choosing any $\beta \in [1-a, 1)$ satisfies both the Orlicz integrability and the strict failure of the fractional condition.
\end{remark}

\section*{}
\setlength{\baselineskip}{18pt}
\addcontentsline{toc}{chapter}{
	{}{\bf


\textbf{Reference}}}


\begin{thebibliography}{99}
	
	\bibitem{bhojak2025alternate}
	A. Bhojak and S. Shrivastava,
	{\it An alternate approach to bilinear rough singular integrals},
	arXiv preprint arXiv:2508.19181 (2025).
	
	\bibitem{BH19}
	E. Buri\'ankov\'a and P. Honz\'ik,
	{\it Rough maximal bilinear singular integrals},
	Collect. Math. \textbf{70} (2019), no.~3, 431--446.
	
	\bibitem{calderon_existence_1952}
	A. P. Calder\'on and A. Zygmund,
	{\it On the existence of certain singular integrals},
	Acta Math. \textbf{88} (1952), 85--139.
	
	\bibitem{calderon_singular_1956}
	A. P. Calder\'on and A. Zygmund,
	{\it On singular integrals},
	Amer. J. Math. \textbf{78} (1956), 289--309.
	
	\bibitem{christ_weak_1988-1}
	M. Christ and J. L. Rubio de Francia,
	{\it Weak type $(1,1)$ bounds for rough operators. II},
	Invent. Math. \textbf{93} (1988), no.~1, 225--237.
	
	\bibitem{coifman_commutators_1975}
	R. R. Coifman and Y. Meyer,
	{\it On commutators of singular integrals and bilinear singular integrals},
	Trans. Amer. Math. Soc. \textbf{212} (1975), 315--331.
	
	\bibitem{coifman_extensions_1977}
	R. R. Coifman and G. Weiss,
	{\it Extensions of Hardy spaces and their use in analysis},
	Bull. Amer. Math. Soc. \textbf{83} (1977), no.~4, 569--645.
	
	\bibitem{connett_singular_1979}
	W. C. Connett,
	{\it Singular integrals near $L^1$},
	Proc. Sympos. Pure Math. \textbf{35}, Part~1 (1979), 163--165.
	
	\bibitem{diestel2011method}
	G. Diestel, L. Grafakos, P. Honz\'ik, Z. Si, and E. Terwilleger,
	{\it Method of rotations for bilinear singular integrals},
	Commun. Math. Anal. \textbf{2011}, Conference~3, 99--107.
	
	\bibitem{dosidis2024multilinear}
	G. Dosidis and L. Slav\'ikov\'a,
	{\it Multilinear singular integrals with homogeneous kernels near $L^1$},
	Math. Ann. \textbf{389} (2024), no.~3, 2259--2271.
	
	\bibitem{dosidis2026bilinear}
	G. Dosidis, B. J. Park, and L. Slav\'ikov\'a,
	{\it Bilinear rough singular integrals near the critical integrability via sharp Fourier multiplier criteria},
	Trans. Amer. Math. Soc., to appear (2026).
	
	\bibitem{Duoandikoetxea}
	J. Duoandikoetxea,
	{\it Fourier analysis},
	translated and revised from the 1995 Spanish original by David Cruz-Uribe,
	Graduate Studies in Mathematics, Vol.~29, Amer. Math. Soc., Providence, RI, 2001.
	
	\bibitem{grafakos1998lp}
	L. Grafakos and A. G. Stefanov,
	{\it $L^p$ bounds for singular integrals and maximal singular integrals with rough kernels},
	Indiana Univ. Math. J. \textbf{47} (1998), no.~2, 455--469.
	
	\bibitem{GT2002_maximal}
	L. Grafakos and R. H. Torres,
	{\it Maximal operator and weighted norm inequalities for multilinear singular integrals},
	Indiana Univ. Math. J. \textbf{51} (2002), no.~5, 1261--1276.
	
	\bibitem{Grafakos2002}
	L. Grafakos and R. H. Torres,
	{\it Multilinear Calder\'on-Zygmund theory},
	Adv. Math. \textbf{165} (2002), no.~1, 124--164.
	
	\bibitem{grafakos2018rough}
	L. Grafakos, D. He, and P. Honz\'ik,
	{\it Rough bilinear singular integrals},
	Adv. Math. \textbf{326} (2018), 54--78.

	\bibitem{GHS2019}	L. Grafakos, D. He, and L. Slav\'ikov\'a, Failure of the H¨ormander kernel condition for multilinear
	Calder´on-Zygmund operators, C. R. Math. Acad. Sci. Paris. \textbf {357} (2019), no. 4, 382-388.
	
	\bibitem{GHHP2024}
	L. Grafakos, D. He, P. Honz\'ik, and B. J. Park,
	{\it On pointwise a.e. convergence of multilinear operators},
	Can. J. Math. \textbf{76} (2024), 1005--1032.
	
	\bibitem{grafakos2020l2}
	L. Grafakos, D. He, and L. Slav\'ikov\'a,
	{\it $L^2\times L^2\to L^1$ boundedness criteria},
	Math. Ann. \textbf{376} (2020), no.~1--2, 431--455.
	
	\bibitem{he2023improved}
	D. He and B. J. Park,
	{\it Improved estimates for bilinear rough singular integrals},
	Math. Ann. \textbf{386} (2023), no.~3--4, 1951--1978.
	
	\bibitem{hofmann_weighted_1988}
	S. C. Hofmann,
	{\it Weighted weak-type $(1,1)$ bounds for singular integrals with nonsmooth kernel},
	Ph.D. thesis, Washington University in St. Louis, 1988.
	
	\bibitem{Kenig1999}
	C. E. Kenig and E. M. Stein,
	{\it Multilinear estimates and fractional integration},
	Math. Res. Lett. \textbf{6} (1999), no.~1, 1--15.
	
	\bibitem{Par25}
	B. J. Park,
	{\it Multilinear estimates for maximal rough singular integrals},
	Ann. Sc. Norm. Super. Pisa Cl. Sci., to appear (2025).
	
	\bibitem{seeger_singular_1996}
	A. Seeger,
	{\it Singular integral operators with rough convolution kernels},
	J. Amer. Math. Soc. \textbf{9} (1996), no.~1, 95--105.
	
	\bibitem{Tao1999}
	T. Tao,
	{\it The weak-type $(1,1)$ of $L\log L$ homogeneous convolution operator},
	Indiana Univ. Math. J. \textbf{48} (1999), no.~4, 1547--1584.
	
\end{thebibliography}
\end{document}